\documentclass[12pt]{amsart}

\pdfoutput=1
\usepackage{latexsym}
\usepackage[centertags]{amsmath}
\usepackage{amsfonts}
\usepackage{amssymb}
\usepackage{tabu}
\usepackage{amsthm}
\usepackage{newlfont}
\usepackage{graphics}
\usepackage{color}
\usepackage{cite}
\usepackage{wrapfig}

\usepackage[usenames,dvipsnames]{xcolor}
\usepackage{graphicx}

\usepackage{wrapfig}

\newtheorem{theo}{Theorem}
\newtheorem{defn}[theo]{Definition}
\newtheorem{exam}[theo]{Example}
\newtheorem{lem} [theo]{Lemma}
\newtheorem{cor}[theo]{Corollary}
\newtheorem{prop}[theo]{Proposition}
\newtheorem{rem}[theo]{Remark}

\newtheorem{prob}[theo]{Problem}

\renewcommand{\arraystretch}{1.3}

\newcommand\xmod[1]{\pmod{\langle #1 \rangle}}
\newcommand\pxmod[3]{#1 \equiv #2  \pmod{\langle #3 \rangle}}

\def\CT{\mathop{\mathrm{CT}}}

\def\char{\textrm{char}}

\def\N{\mathbb{N}}
\def\Z{\mathbb{Z}}
\def\Q{\mathbb{Q}}

\def\x{\mathbf{x}}
\def\m{\mathrm{m}}
\def\y{\mathbf{y}}

\newcommand\mmm[1]{ operations in $(#1)$}

\def\Ryrs{R[y_1,\dots, y_r][[s]]/\langle s^d \rangle}

\title[Generalized Todd Polynomials: Applications to MPA and IP]{Fast Evaluation of Generalized Todd Polynomials: Applications to MacMahon's Partition Analysis and Integer Programming}

\author{
Guoce Xin$^{1}$, Yingrui Zhang$^{2,*}$ and ZiHao Zhang$^{3}$
}

 \address{ $^{1,3}$School of Mathematical Sciences, Capital Normal University,
 Beijing 100048, PR China}
 \address{
  $^2$KLMM, Academy of Mathematics and Systems Science Chinese Academy of Sciences, Beijing 100190, PR China}
 \email{$^1$\texttt{guoce\_xin@163.com}\ \& $^2$\texttt{zyrzuhe@126.com} \& $^3$\texttt{zihao-zhang@foxmial.com}}

\date{ \today }
\thanks{*Corresponding author.}

\begin{document}

\begin{abstract}
The Todd polynomials, denoted as $td_k(b_1,b_2,\ldots,b_m)$, are characterised by their generating functions:
$$\sum_{k\ge 0} td_k s^k = \prod_{i=1}^m \frac{b_i s}{e^{b_i s}-1}.$$
These polynomials serve as fundamental components in the Todd class of toric varieties, a concept of significant relevance in the study of lattice polytopes and number theory. We identify that generalised Todd polynomials emerge naturally within the framework of MacMahon's partition analysis, particularly in the context of computing Ehrhart series.
We introduce an efficient method for the evaluation of generalised Todd polynomials for numerical values of $b_i$. This is achieved through the development of expedited operations in the quotient ring $\mathbb{Z}_p[[s]]$ modulo $s^{d}$, where $p$ is a large prime. The practical implications of our work are demonstrated through two applications: firstly, we facilitate a recalculated resolution of the Ehrhart series for magic squares of order 6, a problem initially addressed by the first author, reducing computation time from 70 days to approximately 1 day; secondly, we present a polynomial-time algorithm for Integer Linear Programming when the dimension is fixed, exhibiting a notable enhancement in computational efficiency.
\end{abstract}

\maketitle

\noindent
\begin{small}
\end{small}

\noindent
\begin{small}
\end{small}

\section{Introduction}\label{S-Introduction}

The Todd class of a toric variety is important in the theory of lattice polytopes and in number theory, and it has been shown that Dedekind sums and their generalizations appear naturally in formulas for the Todd class (see, e.g., \cite{barvinokPommersheim1999summary,cappell1994genera,pommersheim1993toric}). The closely related
Todd polynomials $td_k=td_k(B)$, where $B$ is a multi-set (allow repetitions) of nonzero integers, are defined by their generating function
$$F(s)=\sum_{k=0}^{\infty} td_k(B)s^k=\prod_{b_i \in B} f(b_is)= \prod_{b_i \in B} \frac{b_is}{e^{b_is}-1}.$$ 
They are also called the Todd symmetric polynomials.

In lattice point counting theory, a remarkable result is
Barvinok's ``short" rational representation of the generating polynomial of a rational $\bar{d}$ dimensional polytope $P$:
\begin{align}\label{Barvinok-short-sum}
 f(P;\y)= \sum_{\alpha \in P\cap \Z^{\bar{d}}}  \y^{\alpha} =\sum_{i\in I} \epsilon_i \frac{\y^{u_i}}{\prod_{j=1}^{\bar{d}} (1-\y^{v_{ij}})},
\end{align}
where $\y^\alpha= y_1^{\alpha_1} \cdots y_{\bar{d}}^{\alpha_{\bar{d}}}$, $\epsilon_i\in \{\pm 1\}$. The ``short" here means that the size of the index set $I$ is polynomial in the input size. The Todd polynomials $td_k(B)$ arise in the computation of $f(P;\mathbf{1})$, i.e., the number of lattice points in $P$. See \cite{barvinokwodd2003short,barvinokPommersheim1999summary,verdoolaege2008counting}.
In \cite{de2009ehrhartTodd}, {De Loera, Haws and K{\"o}ppe} show that the technique from Barvinok \cite[Section 5]{barvinok2006computing} can be implemented and obtain a polynomial-time algorithm to evaluate $td_k(B)$ when $k$ is fixed.

There is a parallel development in Algebraic Combinatorics in terms of MacMahon's partition analysis (MPA for short). Indeed MPA applies more generally. Andrews and his coauthors have published a series of 14 papers on MPA beginning with \cite{andrews1998MacMahonfirstpaper}. A CTEuclid algorithm was developed in \cite{xin2015euclid} for complicated combinatorial problems. It succeeds to enumerate magic squares of order $6$.

Generalized Todd polynomial appeared in one important step of CTEuclid, called dispelling the slack variables. This is equivalent to setting $y_i=1$ in \eqref{Barvinok-short-sum} for only part of the $y_i$'s. This is also equivalent to computing generalized Todd polynomial of a special type.

In this paper, we give a fast algorithm for computing $td_k(B)$ and extend it to generalized Todd polynomials, whose generating function is given by
\begin{align*}
 F(s)=e^{as} \frac{\prod_{b \in \bar{B}_0} (1-e^{b s})} {\prod_{b \in B_0} (1- e^{b s})}
\prod_{i=1}^r \frac{\prod_{b \in \bar{B}_i } (1- e^{b s} t_i)} {\prod_{b \in B_i} (1- e^{b s}t_i)},
\end{align*}
where i) $B_i,\bar B_i$ are multi-sets of nonzero rational numbers for $i=0,1,\dots, r$. ii) $t_i\neq 1$ could be distinct numbers, algebraic numbers, variables, or $t_i=t^{m_i}$ for a variable $t$.

As applications, i) we recompute the Ehrhart series of magic squares of order 6, which was first solved by the first named author. The running time is reduced from 70 days to about 1 day; ii) we give a polynomial time algorithm for Integer Linear Programming when the dimension is fixed, with a good performance.

The paper is organized as follows. In Section \ref{s-FastManipulations}, we set up the fast operations in
$\Z_p[[s]]$ modulo $s^{d}$. These include multiplication, division, logarithm and exponential. These operations have quasi linear complexity, i.e., $O(d\log d)$.
In Section \ref{s-FastGToddP}, we give fast computation of generalized Todd polynomials. Basically,
we first compute $\ln F(s)$ modulo $s^{d}$ and then take exponential modulo $s^{d}$.
In Section \ref{s-ConstantTermsGTodd}, we present an algorithm to compute the constant terms of generalized Todd type.
In Section \ref{s-ApplicationMPA}, we give an application to MPA and describe how we recompute the Ehrhart series of magic squares of order 6. Finally, in Section \ref{s-ApplicationILP},
we give a polynomial algorithm for Integer Linear Programming when the dimension is fixed. We report our computational experiments.

\section{Fast Manipulations for $R[[s]]$ Modulo $s^{d}$}\label{s-FastManipulations}
In this section, we focus on operations within the framework of a commutative ring $R$ that encompasses either the field of rational numbers, $\Q$, or the finite Galois field $\Z_p=\Z/(p\Z)$, where $p$ is a prime number. Our objective is to provide efficient algorithms for computations in the quotient ring $R[[s]]/\langle s^{d} \rangle$, which will form the cornerstone of our subsequent developments.

We favor working over $\Z_p$ to circumvent the large integer problem. Consider a scenario where we want to determine a positive integer $N$ with a bounded number of digits through an extended computational process. These might involve expressing $N$ as a sum of rational numbers with long digits. A commonly employed strategy is to compute $N\mod p_i$ for different prime numbers $p_1,\dots, p_k$ such that $p_1\cdots p_k >N$, and subsequently use the Chinese remainder theorem to reconstruct $N$. See Subsection \ref{App-MagicSquares} for an illustration of this approach. A critical consideration is that multiples of $p$ are congruent to zero in $\Z_p$, and hence cannot be divided in $R$ when it contains $\Z_p$.

\subsection{Notation and Seven Operations}\label{ssec-operations}
We introduce the notation for a power series $f(s)$ over $R$, which has coefficients $\{f_n\}_{n\ge 0}$, thus $g(s)=\sum_{n\ge 0} g_n s^n$. Our focus is on operations within the quotient ring $R[[s]]/\langle s^{d} \rangle$, where $\langle s^{d} \rangle$ denotes the ideal generated by $s^{d}$. This is equivalent to performing computations modulo $s^{d}$. Therefore, we have
$$\pxmod{\sum_{n=0}^\infty f_n s^n }{\sum_{n=0}^{d-1} f_n s^n}{s^{d}},  \qquad \text{ or } \sum_{n=0}^\infty f_n s^n
 \xmod{s^{d}} =\sum_{n=0}^{d-1} f_n s^n.$$
We extend our considerations to the polynomial ring $R[\y]$, where $\y=(y_1,\dots,y_r)$, with the understanding that the value of $r$ will be evident from the context. When $r=0$, $R[\y]$ reduces to $R$. We introduce a new concept, which we term ``regularity."
\begin{defn}
Let $f(s,\mathbf{y})\in R[y_1,\dots, y_r][[s]]$, implying that $f_n\in R[\y]$ for all $n$. We say $f(s,\mathbf{y})$ is \emph{regular} in $\mathbf{y}$
if the total degree of $f_n$ in the $y_i$'s, denoted $\deg_{\mathbf{y}} f_n$, is at most $n$ for all $n$. We denote the set of all such $f(s,\mathbf{y})$ by $R[\y]^{reg}[[s]]$.
\end{defn}

We outline seven operations over $R^{reg}[\y][[s]]/\langle s^{d} \rangle$, which implies that the input and output are indeed polynomials in $s$ of degree less than $d$. Such polynomials regular in $\y$ have $
\binom {d+r}{r+1}  = O(d^{r+1}) $ coefficients.

The seven operations are categorized into two groups based on their complexity, which are measured by the number of (ring) operations in $R$:
\begin{enumerate}
\item[G1.] Addition, differentiation, and integration;
\item[G2.] Multiplication, division, exponential, and logarithm.
\end{enumerate}

The G1 operations are relatively straightforward.
\begin{align}
 f(s)\pm g(s) \xmod{s^{d}} &=\sum_{i=0}^{d-1} (f_i\pm g_i) s^i, \tag{addition}\\
 f'(s) \xmod{s^{d}} &=f_1+2f_2s+\cdots +df_{d}s^{d-1},  \tag{derivation}\\
 \int_s f(s)\xmod{s^{d}}&=C+f_0 s+\cdots +f_{d-2} \frac{s^{d-1}}{d-1}, \tag{integration}
\end{align}
where the constant $C$ (with default value $1$) will be clear from the context.
Observe that: i) for addition, we only need $f_i,g_i$ for $0\leq i \leq d-1$;
ii) for derivation, we need $f_i$ for $1\leq i \leq d$;
iii) for integral, we need $f_i$ for $0\leq i \leq d-2$, and the additional
condition $d<\char(R)$ or $\Q\subset R$ to ensure $1/i\in R$.
These operations are performed over the quotient ring $R[y_1,\dots, y_r]^{reg}[[s]]/\langle s^{d} \rangle$. The complexity for these operations is
$O(d^{r+1})$ operations in $R$, with the notable exception that only addition preserves the regularity of the power series.

The G2 operations are more intricate. We have the following result.
\begin{prop}\label{prop-regular}
Suppose $f(s),g(s)\in R[y_1,\dots, y_r]^{reg}[[s]] $. Then if exist,
the G2 operations for $f(s)g(s)$, $g(s)/f(s)$,
$e^{f(s)}$, and $\ln(f(s))$ are
also regular in $\y$.

Furthermore, when operating within the quotient ring $R[y_1,\dots, y_r]^{reg}[[s]]/\langle s^{d} \rangle$, the G2 operations can be accomplished using $O(d^{2r+2})$ \mmm{R} in a conventional manner. In the special case where $R[[s]]/\langle s^{d} \rangle$ is considered, the G2 operations can be executed in $O(d^2)$ \mmm{R}.
\end{prop}
\begin{proof}[Sketched proof]
For the regular property, we only give details for the multiplication $f(s)g(s)$. The other cases
follows similarly by using induction and the recursions in Lemmas \ref{lem-div},\ref{lem-exp},\ref{lem-log} below.

Multiplication is defined by
the traditional convolution formula
\begin{align} \label{e-h=fg}
 f(s)g(s) \xmod{s^{d}} =  \sum_{i=0}^{d-1} h_i s^i, \text{where~} h_i=\sum_{k=0}^{i} f_k g_{i-k}.
 \end{align}
Then it is clear that $h(s) \xmod{s^{d}}$ is uniquely determined by $f(s) \xmod{s^{d}}$
and $g(s) \xmod{s^{d}}$.

Since $f(s)$ and $g(s)$ are regular in $\y$, then $\deg_{\y} f_i\leq i$ and $\deg_{\y} g_i\leq i$ for all $i$.
It follows that $\deg_{\y} h_i \leq \max \{\deg_{\y} f_k+\deg_{\y} g_{i-k}: 0\leq k \leq i\} \leq i$, and hence
$h(s)$ is regular in $\y$.

For the complexity, we will give nicer results in the next subsection. Here
we only consider the $r=0$ case. Thus $f_i,g_i\in R$, and the $O(d^2)$ complexity for $h_0,\dots, h_{d-1}$ is clear by direct expansion.
\end{proof}

For division, exponential, and logarithm we have the following three lemmas.
\begin{lem}\label{lem-div}
If $f(s)$ and $h(s)$ are power series in $R[[s]]$ with $f_0^{-1}\in R$,
then the coefficients of $g(s)=h(s)/f(s)$ can be computed by the recursion
\begin{align}
  \label{e-rec-division}
g_i=f_0^{-1}\Bigg(h_i-\sum_{k=0}^{i-1}f_{i-k} g_{k}\Bigg), \text{with~} g_0=h_0/f_0.
\end{align}

In particular, $g(s) \xmod{s^{d}}$
can be computed from $h(s) \xmod{s^{d}}$ and $f(s) \xmod{s^{d}}$ using $O(d^2)$ \mmm{R} in a conventional manner.
\end{lem}
\begin{proof}
The recursion for $g_i$ is clearly equivalent to Equation \eqref{e-h=fg}. The rest part is easy.
\end{proof}

\begin{lem}\label{lem-exp}
If $h(s)$ is a power series in $R[[s]]$ with $h_0=0$,
then the coefficients of $f(s)=e^{h(s)}$ can be computed by the recursion
\begin{align}\label{e-exp(H)}
f_0&=1; \quad f_n= \frac 1 n  \sum_{i=1}^{n} ih_{i} f_{n-i} , \text{ if } n\geq 1.
\end{align}
In particular, if $d<\char(R)=p$ or $\Q\subset R$, then $f(s) \xmod{s^{d}}$
can be computed from $h(s) \xmod{s^{d}}$ using $O(d^2)$ \mmm{R} in a conventional manner.
\end{lem}
\begin{proof}
Let $f(s)=e^{h(s)}= \sum_{i \geq 0}f_i s^i$ with $f_0=1$.
By taking a derivative, we obtain
$$f'(s)=e^{h(s)}  h'(s)= f(s)  h'(s),$$
which is equivalent to
\begin{align*}
\sum_{n \geq 1} n f_n s^{n-1} = \sum_{k \geq 0} (f_k  s^k) \cdot \sum_{i \geq 0} (h'_i s^{i}).
\end{align*}
Comparing the coefficients on both sides of the above gives
\begin{align*}
f_0&=1; \quad f_n= \frac 1 n  \sum_{i=1}^{n} f_{n-i}h'_{i-1}, \text{ if } n\geq 1,
\end{align*}
which is equivalent to \eqref{e-exp(H)}. This proves the first part of the lemma.

For the second part, observe that $f_n$ is uniquely determined by $h_1,\dots, h_n$.
The condition $d<\char(R)=p$ or $\Q\subset R$ is to ensure $n^{-1}\in R$ for $n < d$.
\end{proof}

\begin{lem}\label{lem-log}
If $f(s)$ is a power series in $R[[s]]$ with $f_0=1$, then the coefficients of
$h(s)=\ln(f(s))$ can be computed by the recursion
\begin{align*}
h_{k}=f_k-\frac{1}{k}\sum_{i=1}^{k-1} i h_i f_{k-i}, \qquad h_0=0, h_1=f_1.
\end{align*}
In particular, if $d<\char(R)=p$ or $\Q\subset R$, then $h(s) \xmod{s^{d}}$
can be computed from $f(s) \xmod{s^{d}}$ using $O(d^2)$ \mmm{R} in a conventional manner.
\end{lem}
\begin{proof}
Let $h(s)=\ln(f(s))$ then $f(s)=e^{h(s)}$. The recursion for $f_k$ in \eqref{e-exp(H)}
can be rewritten as a recursion of $h_{k-1}'=kh_k$ as follows
\begin{align*}
h_{k-1}'=kf_k-\sum_{i=1}^{k-1} f_{k-i} h_{i-1}', \qquad h_0'=f_1.
\end{align*}
The lemma then follows since $h_k$ is uniquely determined by $f_1,\dots, f_k$.
\end{proof}

\subsection{Complexity Analysis}\label{ssec-complexity}
We follow the notation in \cite[Section 8]{von2013modernalgebra}.
A commutative ring $R$ is said to support fast Fourier transform (FFT), if $R$ has a primitive $2^k$th  root of unity for any
nonnegative integer $k$.

Our main objective in this subsection is to prove the following result on the G2 operations.
\begin{theo}
Suppose $R$ is a commutative ring containing $\Q$ or $\Z_p$ for a prime $p>d=2^\ell$.
If $R$ supports FFT, then each of the G2 operations over $R[y_1,\dots, y_r]^{reg}[[s]]/\langle s^{d} \rangle$
can be done using $O(d^{r+1}\log(d^{r+1}))$ \mmm{R}.
\end{theo}

The proof will be divided into some lemmas. Let us start with the $r=0$ case.
\begin{lem}\label{l-FFT}
If $R$ supports FFT, then for given $f(s)\xmod{s^{d}}$ and $g(s)\xmod{s^{d}}$ in $R[s]$, $f(s)g(s)\xmod{s^{d}}$ can be done using
$O(d\log(d))$ \mmm{R}.
\end{lem}
\begin{proof}
Indeed, the polynomial multiplication $\sum_{i=0}^{d-1} f_i s^i \cdot \sum_{i=0}^{d-1} g_i s^i$ can be done using $O(d\log(d))$ \mmm{R} by FFT. For details see \cite[Section 8]{von2013modernalgebra}.
\end{proof}

Denote $M(d)$ to be the time for multiplication of two polynomials in $R[s]$ of degree less than $d$. So $M(d)$ is
the time for $O(d\log(d))$ ring operations in $R$. We simply say $M(d)\in O(d\log(d))$.

\begin{lem}\label{lem-Ndevide&ln}
Let $d= 2^\ell$ for a positive integer $\ell$.
Given $f(s)\xmod{s^{d}} \in R[s]$ with $f_0=1$, we have
\begin{enumerate}
  \item[i)]  $f(s)^{-1} \xmod{s^{d}}$ can be computed in time
$3M( d)+d \in O(d\log(d))$;
  \item[ii)] $\ln f(s) \xmod{s^{d}}$ can be computed in time
$4M(d)+3d \in O(d\log(d))$, provided that $d<\char(R)=p$ or $\Q\subseteq R$.
\end{enumerate}
\end{lem}
\begin{proof}
For part i) we outline Newton's iteration as follows.
Set $g_0(s)=1$ and $g_{i+1}(s)=2g_i(s)-f(s)g_i(s)^2 \xmod{s^{2^{i+1}}}$. Then
$g_\ell(s) = f(s)^{-1} \xmod{s^{2^\ell}}$.
See \cite[Theorem 9.4]{von2013modernalgebra} for details.
Note that there is no restriction on $\char(R)$ since no division over $R$ is involved.

For part ii), recall that all computation are done modulo $s^d$. We first compute the formal derivative $f'(s)$ using $d$ operations in $R$. Next we compute $f(s)^{-1}$ by part i) using $3M(d)+d$ operations.
Then we need $M(d)$ operations for the product $f'(s)f(s)^{-1}=[\ln(f(s))]'$. Finally, we need $d$ operations to obtain $\ln(f(s))=\int [\ln(f(s))]'$ by integration with the condition $d<\char(R)=p$. Thus, the total cost is $4M(d)+3d$.
\end{proof}

\begin{theo}\label{theo-Nexp}
Let $d= 2^\ell$ for a positive integer $\ell$ and $d<\char(R)=p$ or $\Q\subseteq R$.
Given $h(s)\xmod{s^{d}} \in R[s]$ with $h_0=0$, $e^{h(s)} \xmod{s^{d}}$ can be computed in time
$10 M(d)+9d \in O(d\log(d))$.
\end{theo}
\begin{proof}
The process is similar to the proof of Lemma \ref{lem-Ndevide&ln} part i) in \cite[Theorem 9.4]{von2013modernalgebra}.

By Newton's iteration in numerical analysis, we
recursively define
$$ \phi_0(s)=1, \quad  \phi_{i}(s)=\phi_{i-1}(s)- \phi_{i-1}(s) \Bigl(\ln \phi_{i-1}(s) -h(s) \Bigr)=\phi_{i-1}(s)\Big(1- \ln \phi_{i-1}(s)+h(s)\Big) .  $$
We claim that $\ln \phi_i(s)-h(s) \equiv 0  \xmod{s^{2^i}}$,
so that $\phi_\ell(s)\equiv e^{h(s)}\xmod{s^{d}}$.

The claim is proved by induction on $i$.
The basic case $i=0$ is trivial. Assume the claim holds for $i-1$, now
for $i$, we have
\begin{align*}
  \ln \phi_{i}(s)-h(s) &=\ln \phi_{i-1}(s) +\ln \Bigl(1- \ln \phi_{i-1}(s) +h(s)\Bigr)-h(s) \\
                     &= \ln \phi_{i-1}(s)-h(s)+ \sum_{n\ge 1} \frac{(-1)^n}{n} \Big(\ln \phi_{i-1}(s) -h(s)\Big)^n\\
                     &= \sum_{n\ge 2} \frac{(-1)^n}{n} \Big( \ln \phi_{i-1}(s) -h(s)\Big)^n,
\end{align*}
which is clearly $0 \xmod{s^{2^{i}}}$ by the induction hypothesis.

Since $\phi_{i}(s)=\phi_{i-1}(s)\Big(1- \ln \phi_{i-1}(s)+h(s)\Big)\equiv \phi_{i-1}(s)  \xmod{s^{2^{i-1}}}$,
  the lower half of $\phi_{i}(s)$  is equal to  $\phi_{i-1}(s)$ and then we need consider the upper half of $\phi_{i}(s)$.

 The time cost of one iteration of $\phi_{i}(s)=\phi_{i-1}(s)- \phi_{i-1}(s) \Big(\ln \phi_{i-1}(s) -h(s) \Big)$ can be broken down into three steps. First, by Lemma \ref{lem-Ndevide&ln}, we need  $4M(2^i)+3\cdot 2^{i}+2^{i}$  for  computing $\ln \phi_{i-1}(s)-h(s) \xmod{s^{2^i}}$. Second, $M(2^i)$ for the product $\phi_{i-1}(s)\times \big(\ln \phi_{i-1}(s)-h(s)\big) \xmod{s^{2^i}}$. Third, the negative of the upper half of the output in the second step is the upper half of $\phi_{i}(s)$ and this takes $2^{i-1}$ operations.  Thus, we have $5M(2^i)+2^{i+2}+2^{i-1}=5M(2^i)+9 \cdot 2^{i-1}$  for one iteration.
The total cost is then given by
$$\sum_{i=1}^{\ell}5M(2^i) +9 \cdot 2^{i-1} \leq \Big(5M(2^\ell)+9 \cdot2^{\ell-1}\Big) \sum_{i=1}^{\ell} 2^{i-\ell} \leq 10M(2^\ell)+9\cdot 2^\ell=10 M(d)+9d,$$
where we have used the fact that $2M(n) \leq M(2n)$ for all $n \in \N$.
\end{proof}

The complexity for operations in $\Ryrs$ is much more complicated. We use three Kronecker substitutions:
(i) $\Gamma_\y:  y_i\to y_1^{d^{i-1}}$ for $2\le i \le r$;
(ii) ${\Gamma}_s: s \to y_1^{d^r}$;
(iii) $\overline{\Gamma} = {\Gamma}_s \Gamma_{\y}$,
or equivalently,
$\overline{\Gamma}:  y_i\to y_1^{d^{i-1}}$ for $2\le i \le r+1$, where we set $s=y_{r+1}$.

\begin{lem}\label{lem-Gamma}
If $f(s)$ in $R[y_1,\dots, y_r][[s]]$ is regular in $\y$, then
we can reconstruct $f(s) \xmod{s^d}$ from
$\Gamma_\y f(s) \xmod{s^d}$ or
 $\overline{\Gamma} f(s)\xmod{y_1^{d^{r+1}}} $ in time $O(d^{r+1})$.
\end{lem}
\begin{proof}
Write
$g(s)=\Gamma_\y f(s)=\sum_{i=0}^{d-1} g_i s^i$ and $h=\Gamma_s g(s)$. Then for $i < d$ we have
$$0\leq \deg_{y_1} g_i \leq \deg_{y_1} \Gamma_\y y_r^{i} = id^{r-1} \Longrightarrow
                                    i d^r \leq \deg_{y_1} \overline{\Gamma} g_is^i \leq i(d^{r-1}+d^r)<(i+1)d^{r}.$$
Observe that $\deg_{y_1} \overline{\Gamma} g_is^i$ is less than $d^{r+1}$ for $i<d$ and is greater than or equal to $d^{r+1}$ for $i\geq d$.
It follows that
$\overline{\Gamma}: \Ryrs \to R[[y_1]] /\langle y_1^{d^{r+1}} \rangle $ is an injection.

Give $h$ as input, $g(s)$ can be recovered by $g(s)= {\Gamma}^{-1}_s h$, obtained by
$y_1^\beta\mapsto y_1^{\gamma_1} s^{\gamma_2}$ where $\beta= \gamma_2 d^r+\gamma_1$, $0\leq \gamma_1< d^r$. Similarly, to obtain $f(s)$ we use $y_1^\beta \mapsto y_1^{\beta_1}\cdots y_{r+1}^{\beta_{r+1}}$ in $h$, where $\beta_1+\beta_2 d+\cdots + \beta_{r+1} d^r$ is the base $d$ representation of $\beta$.
The cost is at most $O(d^{r+1})$ operations in $R$.
\end{proof}
Note that $\overline{\Gamma}$ is not injective in general as $\overline{\Gamma} s^{d+1} = \overline{\Gamma} s^{d} y_r^d$.

Using $\overline{\Gamma}$ in Lemma \ref{lem-Gamma}, we can give the complexity
for multiplications and divisions as follows.
\begin{cor}
Give $f(s), g(s)$ in $\Ryrs$ as input, if they are both  regular in $\y$, then $ f(s) g(s) \xmod{s^d}$ can be computed in time $M(d^{r+1})  \in O((r+1)d^{r+1}\log(d)).$
\end{cor}

\begin{cor}
Give $f(s)$ in $\Ryrs$ as input, if it is regular in $\y$ and $f_0=1$, then $f(s)^{-1} \xmod{s^d}$ can be computed in time $ O((r+1)d^{r+1}\log(d) ).$
\end{cor}
Note that there is no restriction on $p=\char(R)$, since the computation does not involve division over $R$.

For logarithm and exponential, we need to use both $\overline{\Gamma}$ and $\Gamma_\y$.

\begin{cor}\label{cor-logrgeq1}
Suppose $\char(R)=p >d$ or $\Q\subset R$. Give $f(s)$ in $\Ryrs$ as input, if it is regular in $\y$ and $f_0=1$, then $\ln(f(s)) \xmod{s^d}$ can be computed in time $ O((r+1)d^{r+1}\log(d) ).$
\end{cor}
\begin{proof}
We use the formula $ \Gamma_\y \ln(f(s)) =\int_s  \Big(\frac{d}{ds} f_1(s)\cdot f_1(s)^{-1}\Big)$ where $f_1(s) = \Gamma_\y f(s)$. To carry out the idea, we outline as follows.
$$f(s)\to f_1(s) \to f_2(s)=\frac{d}{ds} f_1(s) \to
f_3(s)=f_2(s)/f_1(s)\to \ln(f(s))=\Gamma_\y^{-1}\int_s f_3(s).$$
We shall use $\bar{f_i}(y_1)=\overline{\Gamma} f_i(s)$.

The computation of $f_1(s)$, $f_2(s)$ cost at most $O(d^{r+1})$ time, and so does the last step. Note that the condition $p>d$ is used in the integration.

To compute $f_3(s)$ we compute
$$\overline{\Gamma} f_3(s)= (\overline{\Gamma} f_2(s)) (\overline{\Gamma} f_1(s) )^{-1}
\xmod{y_1^{d^{r+1}}} $$
and then apply $\Gamma_s^{-1}$.
The time cost is about $3M(d^{r+1})$ for reciprocal, another $M(d^{r+1})$ for multiplication,
and $O(d^{r+1})$ for applying $\Gamma_s^{-1}$.
Therefore, the total running time is $O\Big(d^{r+1} \log (d^{r+1})\Big)$, as desired.
\end{proof}

\begin{rem}
Basically, we shall not use $\overline{\Gamma}$ in integration, logarithm, and exponential. More precisely, we shall not use
the outline in the first version of this paper:
$$f(s) \to f_1(y_1) = \overline{\Gamma} f(s) \to f_2(y_1)=\frac{f_1'(y_1)}{f_1(y_1)} \to
f_3(y_1) = \int_{y_1} f_2(y_1) \to \ln(f(s))= \overline{\Gamma}^{-1} f_3(y_1).$$
This is because the integration $\int_{y_1} f_2(y_1)$ may produce multiple of $p$ in the denominator when $p=\char(R)$ is less than $d^{r+1}$.
\end{rem}

\begin{theo} \label{theo-RegularExp}
Suppose $\char(R)=p >d$ or $\Q\subset R$. Given $h(s)\in \Ryrs$, if $h_0=0$
and $h(s)$ is regular in $\mathbf{y}$, then
$e^{h(s)}\xmod{s^d}$ is also regular in $\mathbf{y}$, and it can be computed in time
$O((r+1)d^{r+1}\log(d))$.
\end{theo}
\begin{proof}
Let $f(s)=e^{h(s)}= \sum_{i \geq 0}f_i s^i$ with $f_0=1$ and $f_i \in R[y_1,\dots, y_r]$ for $i \geq 1$.
By Lemma \ref{lem-exp}, we have $f_n = \frac{1}{n}\sum_{k=1}^n f_{n-k}h'_{k-1}$.
Then the regularity of $f(s)$, i.e., $\deg_\y f_n\leq n$, follows by induction on $n$.

By Lemma \ref{lem-Gamma}, it suffices to compute $\Gamma_\y f(s)$. We use a similar process as in
the proof of Theorem \ref{theo-Nexp}.

Let $d=2^\ell$. We iteratively compute by
$$ \phi_0(s)=1, \quad   \phi_{i}(s) = \phi_{i-1}(s)\Big(1- \ln \phi_{i-1}(s)+\Gamma_\y h(s)\Big) \xmod{s^{2^i}}.  $$
Then $\phi_\ell(s)=\Gamma_\y f(s)$.

In the $i$-th iteration, since $\deg_{y_1} \Gamma_s \phi_i(s) <2^i d^r$, the computation of $\ln \phi_{i-1}(s) \xmod{s^{2^i}}$ cost $O(2^{i} d^r\log(2^{i}d^r)) $ operations in $R$. Next we
compute $\overline{\Gamma} \phi_i(s) = \overline{\Gamma} \phi_{i-1}(s) \cdot
\overline{\Gamma} (1- \ln \phi_{i-1}(s)+\Gamma_\y h(s)\Big) \xmod{s^{2^i}} $,
which cost another $O(2^{i} d^r\log(2^{i}d^r))$,
and then apply $\Gamma_s^{-1}$ to obtain $\phi_i(s)$. Thus the total cost in the $i$-th
iteration is $O(2^{i} d^{r}\log(2^{i} d^{r})) $.

A similar analysis as in the proof of Theorem \ref{theo-Nexp} gives the desired complexity.
\end{proof}

Though it is common in computer algebra to use ring operations in $\Q$, the time for one operation in $\Q$ depends on the binary encoding length of rational numbers.
This problem is also referred to as the large integer problem. In particular, we report that the quasi-linear complexity in Theorem \ref{theo-RegularExp} seems not achievable over $\Q$, according to our computer test. It is possible to do complexity analysis taking into
account of the binary encoding length of rational numbers. This type of analysis appeared in \cite{de2009ehrhartTodd}.

We favor working modulo a prime $p$, because it is an appropriate unit for one operation in $\Z_p$, or one operation in $\Z_p[\alpha]$ where $\alpha$
is algebraic over $\Z_p$. On the opposite, it is clearly not an appropriate unit for one operation in $\Z_p[y]$ when $y$ is a variable.


\subsection{The log-exponential trick}
Our algorithm for computing generalized Todd polynomials benefits from the log-exponential trick for the product of similar polynomials. This technique is demonstrated through three examples, with a particular emphasis on Example \ref{e-logexpspeedup}. We always assume the  Fast Fourier Transform (FFT) can be performed.

We commence with the rapid computation of power sums for a given multiset \( B = \{b_1, \dots, b_k\}^* \) (the $^*$ means allowing repetitions)
defined by \( p_n(B) = b_1^n + \cdots + b_k^n \).
\begin{lem}\label{lem-powersumcomplexity}
Suppose \( R \) is a commutative ring with a unit that supports the FFT. Given a multiset \( B = \{b_1, \dots, b_k\}^* \) of $k$ elements in \( R \), the \( n \)-th power sum \( p_n(B) \) for \( n = 1, \dots, d-1 \) can be computed using $O(k\log^2(d)+d\log(d))$ \mmm{R}.
\end{lem}
\begin{proof}
Let \( f(s) = (1 - b_1s)(1 - b_2s) \cdots (1 - b_ks) \). Note that
$$ -\ln(f(s)) = \sum_{j=1}^{k} \ln\left(\frac{1}{1 - b_js}\right) = \sum_{n \geq 1} \frac{1}{n} p_n(B) s^n. $$
We first compute $f(s) \xmod{s^{d}}$ using $O(k \log^2(\min(k,d))) $ \mmm{R} as we shall discuss below, and then
compute $-\ln (f(s)) \xmod{s^{d}}$ in $O(d \log(d))$ \mmm{R} by Lemma \ref{lem-Ndevide&ln} ii). This yields
the complexity of \( O(k \log^2(d)+d\log(d)) \) \mmm{R} if $k>d$
and $O(k\log^2(k)+d\log(d))$ \mmm{R} if $k\leq d$. The lemma then follows.

The coefficients \( f_i \) of \( f(s) \) can be computed using \( O(k \log^2(k)) \) \mmm{R} by \cite[Lemma 10.4]{von2013modernalgebra}.
This is already correct when $k\leq d$. If $k>d$, then we divide $B$ into a disjoint union $B_1 \uplus B_2 \uplus \cdots \uplus B_\ell$ (of multisets),
where \( \ell = \lceil k/(d-1) \rceil \) and each \( |B_i|\leq d-1 \). It follows that $p_n(B)=p_n(B_1)+\cdots + p_n(B_\ell)$ for $n<d$
can be computed using \( O(k/d \cdot d\log^2(d) + k/d \cdot d) = O(k \log^2(d)) \) \mmm{R}.
\end{proof}

\begin{cor}\label{cor-add-similar}
Consider a multiset \( B \) of \( k \) elements within the ring \( R \) as previously defined. If \( h(s) \) is provided in \( R[y][[s]] \xmod{s^d} \) and is regular with respect to \( y \) with \( h_0 = 0 \), then the computation of \( H(s) = \sum_{b \in B} h(b s) \) can be achieved using \( O(k \log^2(d) + d^2) \) \mmm{R}. Furthermore, if \( h(s) \) is independent of \( y \), the computation of \( H(s) \) can be performed using \( O(k \log^2(d) + d\log(d)) \) \mmm{R}.
\end{cor}
\begin{proof}
We have \( H(s) = \sum_{n=1}^{d-1} p_n(B) h_{n} s^n \). Consequently, we can first calculate \( p_n(B) \) for \( n \leq d-1 \) using \( O(k\log^2(d) + d\log(d)) \) \mmm{R} by Lemma \ref{lem-powersumcomplexity}. Given that \( \deg_y h_n \leq n \), we require an additional \( 2 + 3 + \cdots + d = O(d^2) \) \mmm{R} to obtain \( H(s) \). In the special case where \( h(s) \) is free of \( y \), only an additional \( d \) \mmm{R} are necessary. The corollary follows from these observations.
\end{proof}

\begin{exam}\label{e-logexpspeedup}
Given $f(s) \xmod{s^{d}} \in \Z_p[s]$ where $p>d$ and $f_0=1$. Compute $g(s)=\prod_{i=1}^k f^i(s) \xmod{s^{d}}$, where
$f^i(s)=f(b_i s)$ for $b_i\in \Z_p$, and $k\geq d$.
\end{exam}
\begin{proof}[Solution]
In the conventional approach, the computation of $f^i(s)=f(b_i s) \xmod{s^{d}}$ for all $i$ necessitates \(2kd\) operations in \(\mathbb{Z}_p\). Subsequently, the \(k-1\) polynomial multiplications required for the computation of $g(s) = \prod_{i=1}^k f(b_i s) \xmod{s^{d}}$ incur a cost of \((k-1)O(d\log(d)) = O(kd\log(d))\) operations in \(\mathbb{Z}_p\) (as per Lemma \ref{l-FFT}). Consequently, the total cost of this straightforward method is \(O(kd\log(d))\) operations in \(\mathbb{Z}_p\).

We can expedite the computation by employing the log-exponential trick.
Let $h(s)=\ln (f(s))$. Then
$$\ln g(s)=H(s) = \sum_{i=1}^k h(b_i s)= \sum_{n\ge 0} h_n p_n(B) s^n,$$
where $p_n(B)=b_1^n+\cdots+b_k^n$ is the $n$-th power sum.

We commence by calculating \(h(s) \xmod{s^{d}}\) in \(O(d\log(d))\) operations in \(R\) (as per Lemma \ref{lem-Ndevide&ln} ii). Subsequently, we compute $H(s)$ using \(O(k\log^2(d)+d\log(d))\) operations in \(R\) (as per Corollary \ref{cor-add-similar}). Finally, the computation of the exponential \(g(s)=e^{H(s)} \xmod{s^{d}}\) requires \(O(d\log(d))\) operations in \(R\). This aggregates to a total cost of only \(O(k\log^2(d)+d\log(d))\) operations in \(R\).

It is important to note that the log-exponential trick does not offer any efficiency gains for computing $g(s)=f^1(s)\cdots f^k(s)$
if there is no inherent relationship among the \(f^i(s)\). In such a scenario, the computation of \(\ln(f^i(s)) \xmod{s^{d}}\) for \(i = 1, \dots, k\) would already necessitate \(O(kd\log(d))\) operations in \(R\).
\end{proof}

\begin{exam}
Compute the constant term in \( \Q((q)) \) for the following expression:
$$\CT_q \frac{1}{q^c \prod_{j=c+1}^{d} (1 - m_j q)} = [q^c] \frac{1}{\prod_{j=c+1}^{d} (1 - m_j q)}.$$
This constant term arises in the context of polytope volume computation as discussed in \cite{Volume2023}, where it is necessary to evaluate
$$\CT_q \frac{1}{\prod_{i=1}^{d} (a_i - b_iq)} = \CT_q \frac{(-1)^c}{q^c \prod_{i=1}^{c} b_i \prod_{j=c+1}^{d} a_j (1 - b_j/a_j q)}.$$
\end{exam}
\begin{proof}[Solution]
The computation of the polynomial \( \prod_{j=c+1}^{d} (1-m_j q) \) modulo \( q^{c+1} \) can be executed in \( O((d-c)\log^2(c)) \) operations in \( \mathbb{Q} \), as established in the proof of Lemma \ref{lem-powersumcomplexity}. Following this, determining the reciprocal modulo \( q^{c+1} \) entails an additional \( O(c\log(c)) \) operations in \( \mathbb{Q} \) (per Lemma \ref{lem-Ndevide&ln} i)). Hence, the total computational complexity is \( O((d-c)\log^2(c)+c\log(c)) \) operations in \( \mathbb{Q} \).
\end{proof}

\begin{exam}
Consider a sequence of positive integers \(a_0, a_1, \ldots, a_k\). If \(a_0\) is of order \(O(2^k)\), then the computation of the constant term
$$\CT_q \frac{q^{-a_0}}{(1-q)(1-q^{a_1})\cdots (1-q^{a_k})}$$
can be achieved in \(O(a_0 \log(a_0))\) time complexity. This efficiency is particularly beneficial when \(a_0 < 2^{16}\) and \(k < 16\).
\end{exam}
\begin{proof}[Solution]
The computation proceeds in several steps. First, we calculate the polynomial \(f(q) = (1-q^{a_1})\cdots (1-q^{a_k})\) in \(O(2^k)\) time, considering only the terms up to \(q^{a_0}\) as the higher-order terms will not contribute to the constant term.
Next, we determine \(g(q) \mod{q^{a_0+1}}\) in \(O(a_0\log(a_0))\) time, where \(g(q)\) is the reciprocal of \(f(q)\). This step is efficient due to the fact that we are working modulo \(q^{a_0+1}\), which limits the degree of the polynomial.
The constant term sought is the sum of the first \(a_0+1\) coefficients of \(g(q)\). Consequently, the overall running time is \(O(a_0 \log(a_0) + 2^k + a_0)\), which simplifies to \(O(a_0 \log(a_0))\) since \(2^k\) is subsumed by \(a_0\) under the given constraints. This confirms the desired time complexity for the computation.
\end{proof}

\section{Fast computation of generalized Todd polynomials}\label{s-FastGToddP}

Let $a\in \Q$, and let $B_0,\bar B_0,B_1,\bar B_1,\dots, B_r, \bar B_r$ be finite multi-sets of nonzero integers, with cardinalities $k_i,\bar{k}_i$ respectively. The generalized Todd polynomials $$gtd_n:=gtd_n(a,B_0,\bar B_0,B_1,\bar B_1,\dots, B_r, \bar B_r)$$ are defined by their generating function
\begin{align}\label{e-GTodd}
  F(s)=\sum_{n\ge 0} gtd  _n s^n&=e^{as} \frac{ \prod_{b\in B_0} f(bs)}{\prod_{b\in \bar B_0} f(b s)}
  \prod_{i=1}^r  \frac{\prod_{b\in B_i} f(bs,y_i)} {\prod_{b\in \bar B_i} f(bs,y_i)},
\end{align}
where
\begin{align*}
  f(s)= \frac{s}{e^s-1}=1+o(1), \qquad  f(s,y)= \frac{1}{1-y (e^s-1)}=1+o(1).\\
\end{align*}

By making use of the Stirling numbers of the second kind $S(n,k)$, and the Bernoulli numbers $\mathcal{B}_{k}$ defined by
$$\frac{(e^x-1)^k}{k!} = \sum_{n\ge k} S(n,k) \frac{x^n}{n!} \qquad \qquad \text{ and } \qquad \qquad \frac{x}{e^x-1} = \sum_{n\ge 0} \frac {\mathcal{B}_{n}} {n!} x^n    ,$$
we can give explicit formulas of $f(s)$ and $f(s,y)$. In fact, what we need are the formulas of $\ln f(s)$ and $\ln f(s,y)$ given by
\begin{equation}\label{e-EM}
  \begin{aligned}
 \ln f(s)=h(s)&= -\sum_{n\geq 1} \frac{\mathcal{B}_{n}}{n\cdot n!}s^{n}=-{\frac{1}{2}}s-{\frac{1}{24}}{s}^{2}+{\frac {1}{2880}}{s}^{4}+O
 \left({s}^{6}\right),\\
 \ln f(s,y)=h(s,y) & = -\ln(1-y (e^s-1)) =\sum_{k\geq 1}  \frac{y^{k}(e^s-1)^k}{k} =\sum_{k\geq 1}y^k \sum_{n\geq k} S(n,k) \frac{(k-1)!}{n!} s^n\\
    &=\sum_{n\geq 1}\left( \sum_{k=1}^n \frac{(k-1)!}{n!} S(n,k)y^k \right) s^n= \sum_{n\geq 1}  C_{n}(y) s^{n},
\end{aligned}
\end{equation}
where
\begin{align}\label{e-Cn(y)}
C_{n}(y)&=  \sum_{k=1}^n \frac{(k-1)!}{n!} S(n,k)y^k
\end{align}
is a polynomial of degree $n$ in $y$, with $C_n(0)=0$. Therefore
$f(s,y)$ is regular in $y$.

The above definition extends for ring $R$ of a more general form, but
we insist that $R$ is either $\Q$ or $\Z_p$, which is sufficient for our purpose.

\begin{lem}\label{l-hsandhsy}
Suppose \(R\) is either \(\Q\) or \(\Z_p\) with \(p>d\).
Let \(h(s)\) and \(h(s,y)\) be defined as in \eqref{e-EM}. Then
\(h(s) \mod{s^d}\) can be computed using \(O(d\log(d))\) operations in \(R\);
\(h(s,y) \mod{s^d}\) can be computed using \(O(d^2\log(d))\) operations in \(R\).
\end{lem}
\begin{proof}
The computation of \(h(s) \mod{s^d}\) begins with the construction of the following expression using \(O(d)\) operations in \(R\):
\begin{align*}
  \pxmod{(e^s-1)/s}{\sum_{n=1}^{d-1} \frac{s^{n-1}}{n!}}{s^{d}}.
\end{align*}
Thereafter, \(\ln f(s)\) can be determined using Lemma \ref{lem-Ndevide&ln} with an additional \(O(d\log(d))\) operations in \(R\).

For \(h(s,y)\), we initiate the computation by constructing the following expression using \(O(d)\) operations in \(R\):
\begin{align*}
  \pxmod{1-y(e^s-1)}{1-\sum_{n=1}^{d-1} \frac{y}{n! }s^{n}}{s^{d}}.
\end{align*}
Subsequently, \(\ln f(s,y)\) can be calculated using Corollary \ref{cor-logrgeq1} with a further \(O(d^2\log(d^2))\) operations in \(R\), which simplifies to \(O(d^2\log(d))\).
\end{proof}

\begin{theo}\label{theo-GTodd}
Suppose \(R\) is either \(\Q\) or \(\Z_p\) with \(p > d\).
Given \(a\), \(B_0,\bar B_0,B_1,\bar B_1,\dots, B_r, \bar B_r\) and \(F(s)\) as defined in \eqref{e-GTodd},
if \(r\geq 1\) then we can compute the sequence \((gtd_0,gtd_1,\dots, gtd_{d-1})\) of generalized Todd polynomials
using \(O\Big((r+1)d^{r+1}\log(d)+ \log^2(d)\sum_{i=0}^{r} (|B_i|+|\bar B_i|) \Big)\) operations in \(R\).
\end{theo}
\begin{proof}
Let \(h(s;B)=\sum_{b\in B} h(bs)\) and \(h(s,y;B)=\sum_{b\in B} h(bs,y)\). Then
\begin{equation}
  \ln(F(s))=H(s)=as + h(s;B_0) - h(s;\bar B_0) + \sum_{i=1}^r h(s,y_i;B_i)-h(s,y_i;\bar B_i) . \label{eq-diffH}
\end{equation}
We proceed by constructing \(H(s) \mod{s^d}\) and subsequently computing \(F(s)\equiv e^{H(s)} \mod{ s^{d}}\),
yielding the desired sequence \((gtd_0,gtd_1,\dots, gtd_{d-1})\).

In Step 0, we utilize Lemma \ref{l-hsandhsy} to compute \(h(s)\) and \(h(s,y)\) with \(O(d^2\log(d))\) operations in \(R\).

In Step 1, we apply Corollary \ref{cor-add-similar} to compute \(h(s;B_0)-h(s;\bar B_0) \mod{s^{d}}\)
using \(O\Big(d\log(d)+\log^2(d)(|B_0|+|\bar B_0|)\Big)\) operations in $R$, and similarly for \(h(s,y_i;B_i)-h(s,y_i;\bar B_i) \mod{s^{d}}\) for \(1 \leq i\leq r\), using \(O\Big(rd^2 +\log^2(d)\sum_{i=1}^{r} (|B_i|+|\bar B_i|)\Big)\) operations in $R$. The combined cost for this step is
\(O\Big(r d^2+ d\log(d)+\log^2(d)\sum_{i=0}^{r} (|B_i|+|\bar B_i|)\Big) \).

In Step 2, we invoke Theorem \ref{theo-RegularExp} to compute \(e^{H(s)} \mod{s^{d}}\). The operational cost for this step
amounts to \(O\Big((r+1)d^{r+1}\log(d)\Big)\).

The total complexity is the sum of the complexities of the individual steps, yielding the asserted bound.
\end{proof}

The case \( r = 0 \) is handled distinctly. We give the following Theorem.

\begin{theo}
\label{theo-cor-todd}
With the notation established in Theorem \ref{theo-GTodd}, and considering the specific case where \( r = 0 \), the function \( F(s;a,B_0,\bar B_0) \) is given by
$$ F(s;a,B_0,\bar B_0) = e^{as}\prod_{b\in B_0} f(bs) \prod_{b\in \bar B_0} f(b s)^{-1}. $$
When \( a = \frac{p_1(B_0)-p_1(\bar B_0)}{2} \), the sequence \( (td_0,td_1,\dots, td_{d-1}) \pmod{p} \) can be computed in time \( O\Big(d\log(d)+\log^2(d)(|B_0|+|\bar{B}_0|)\Big) \).
\end{theo}
Note that the case with an empty \( \bar{B}_0 \) has been treated in \cite{de2009ehrhartTodd} over \(\mathbb{Q}\). Our result offers a more efficient computation.
\begin{proof}
By substituting \( a \), \( B_0 \), and \( \bar B_0 \) into \eqref{eq-diffH}, we obtain
\begin{align}
  H(s)\xmod{s^{d}}&=\frac{p_1(B_0)-p_1(\bar B_0)}{2} s+ \sum_{n=1}^{d-1} \bigg( \frac{\mathcal{B}_{n}}{n \cdot n !} \left(p_{n}(B_0)-p_{n}(\bar B_0)\right)
 \bigg)  s^{n} \nonumber \\
 &=\sum_{n=1}^{\lfloor \frac{d-1}{2} \rfloor} \bigg( \frac{\mathcal{B}_{2n}}{2n\cdot(2n)!} \left(p_{2n}(B_0)-p_{2n}(\bar B_0)\right)\bigg)  s^{2n}.\label{e-H-half}
\end{align}
Here, we exploit the fact that \( \mathcal{B}_{2r+1} = 0 \) for \( r > 0 \), indicating that \( H(s) \) is an even function, and thus \( td_n=0 \) for all odd \( n \). While this property does not reduce the complexity, we proceed with a variable substitution, setting \( \bar{s} = s^2 \), so that \( \bar{H}(\bar{s}) = H(s) \).

To construct \( \bar{H}(\bar{s}) \xmod{\bar s^{d'}} \), where \( d'=\lfloor d/2 \rfloor \),
we must first compute \( B_0^2=\{b^2: b\in B_0\}^* \) and \( \bar{B}_0^2 \) using \( |B_0|+|\bar{B}_0| \) operations in $R$. Then,
we calculate \( p_{2n}(B_0)=p_n(B_0^2) \) and \( p_{2n}(\bar{B}_0)=p_n(\bar B_0^2) \) for \( n < d' \) using \( O\Big((|B_0|+|\bar{B}_0|)\log^2(d')+2d'\log(d')\Big) \) operations in $R$.
Applying the relevant formula \eqref{e-H-half} requires \( O(d') \) operations in $R$. The overall cost for this step is
\( O\Big((|B_0|+|\bar{B}_0|)\log^2(d') +2d'\log(d')\Big) \).

Subsequently, applying Theorem \ref{theo-Nexp} with respect to \( \bar{s} \) yields \( e^{\bar{H}(\bar{s})} \xmod{\bar s^{d'}} \), which corresponds to \( F(s) \) after reverting to the original variable \( s \) via \( \bar{s} = s^2 \).
This step incurs a cost of \( O(d'\log (d')) \). The theorem follows directly from these computations.
\end{proof}

\section{Constant Terms of generalized Todd Type}\label{s-ConstantTermsGTodd}
For a Laurent series $G(s)=\sum_{n\ge N} G_n s^n$ over a ring $R$, i.e.,
$G_n\in R$, the order of $G(s)$ is defined to be
$\theta_s(G(s))=\min \{ n: G_n\neq 0\}$. When the $s$ is clear from the context, we use $\theta$ as short for $\theta_s$.
For instance, $\theta(2s^{-2}+1+3s+\cdots)=-2$, and $\theta(s+5s^2+\cdots)=1$.
Let $d=-\theta(G(s))$, then $G(s)s^{d}$ is a power series.
In other words, $-\theta(G(s))$ is the smallest integer $d$ such that $G(s)s^{d}$ is a power series.

We will use the following easy facts without mentioning: i) $\theta (G(s)\pm  H(s))\geq \min(\theta(G(s)),\theta(H(s)))$; ii) $\theta (G(s) H(s))=\theta(G(s))+\theta(H(s))$;
iii)
$\theta(G(s)^{-1})=-\theta(G(s))$.

The constant term of $G(s)$ is denoted $\CT_s G(s)=G_0$. In particular, if $\theta(G(s))\ge 0$ then $\CT_s G(s)=G(0)$, and we call this case \emph{trivial}.
Thus we shall focus on the $d=-\theta(G(s))>0$ case.  In this section, we consider $\CT_s G(s)$ for $G(s)$ of the following
generalized Todd Type
\begin{align}\label{eq_Gtype}
G(s)= L(e^s) \frac{\prod_{b \in \bar{B}_0} (1-e^{b s})} {\prod_{b \in B_0} (1- e^{b s})}
\prod_{i=1}^r \frac{\prod_{b \in \bar{B}_i } (1- e^{b s} t_i)} {\prod_{b \in B_i} (1- e^{b s}t_i)},
\end{align}
where i) $t_i\neq 1$ are distinct elements;
ii) $L(\kappa)$ is a linear combination of $U$ monomials written as $L(\kappa)=\sum_{i=1}^U \ell_i  \kappa^{a_i}$ where $a_i$ are rational numbers;
iii) $B_i,\bar B_i$ are multi-sets of nonzero rational numbers for $i=0,1,\dots, r$. When modulo a prime $p$, rational numbers should be valid elements in $\Z_p$.

Before going further, let us briefly explain these technical conditions.

\begin{enumerate}
  \item In lattice point counting theory, $L$ is usually a monomial. We allow $L$ to be a Laurent polynomial because it is standard in constant term world (or in MPA), and it also applies to lattice point counting in terms of  K{\"o}ppe's primal decomposition \cite{PrimalBarvinok}.

  \item For lattice point counting, we only have $B_0$, with all the other multi-sets being empty. We add $\bar{B}_0$ because
it appears in a recent work \cite[Example 5.5]{Fliuxin23} on Frobenius problems, especially in the computation of Sylvester number.
In that example, one has
\begin{align*}
f(x)&=1+\frac{1-x^{dk}}{1-x^d}\cdot \frac{x^{ha+d}(1-x^{(ha+dk)s})}{1-x^{ha+dk}}
+\frac{x^{(s+1)ha+d(sk+1)}(1-x^{dr_1})}{1-x^d},
\end{align*}
where $a,d, h,k,s$ are integer parameters, and one needs $f'(x)|_{x=1}$. By writing
\begin{align*}
f'(e^t)&=\frac{dke^{t(ah+d(k+1))}\left(1-{{e}^{t((ah+dk)s)}}\right)}{{{e}^{t}}\left(1- {{e}^{td}}\right)\left(1-{{e}^{t(ah+dk)}}\right)}+ \cdots,
\end{align*}
one needs to compute its limit at $t=0$.

\item The appearance of $B_i$ is natural in MPA as we shall discuss in Section \ref{s-ApplicationMPA}. The appearance of $\bar{B}_i$
arose in the computation of the constant term of
  \begin{align*}
  F(d) &=  \frac{1}{\prod_{i=0}^{d} x_i(1-x_i)} \cdot \prod_{0\le i < j \le d} \frac{1 - 3 x_j/x_i + (x_j/x_i)^2}{(1-x_j/x_i)^2}.
  \end{align*}
This problem was asked by Zhicong Lin \cite{LinProblem}, where one shall treat $1>x_0>x_1>\dots >x_d>0$ for series expansion.
The case $d=1$ already explains the appearance of $\bar B_1$. Let $\alpha, \alpha^{-1} $ be the two roots of $y^2 - 3y + 1 = 0$. Then
we have
\begin{align*}
  \CT_{\x} F(1) &= \CT_{\x} \frac{x_0^{-1} x_1^{-1} (1 - \alpha x_1/x_0) (1 - x_1/(\alpha x_0))}{(1 - x_0) (1 - x_1) (1 - x_1/x_0)^2}\\
   &= \CT_{x_0} \CT_{x_1} \frac{x_0^{-1} x_1^{-1} (1 - \alpha x_1/x_0) (1 - x_1/(\alpha x_0))}{(1 - x_0 z_1) (1 - x_1 z_2) (1 - z_3 x_1/x_0) (1 - z_4 x_1/x_0)} \Bigg|_{z_i =1}\\
   &= \frac{z_1 z_2 (1 - \alpha z_1 z_2^{-1}) (1 - z_1/(\alpha z_2))}{(1 - z_1 z_2^{-1} z_3) (1 - z_1 z_2^{-1} z_4)} \\
   &\quad + \frac{z_2^2 z_3 (1 - \alpha z_3^{-1}) (1 - \alpha^{-1} z_3^{-1})}{(1 - z_1^{-1} z_2 z_3^{-1}) (1 - z_3^{-1} z_4)}
   + \frac{z_2^2 z_4 (1 - \alpha z_4^{-1}) (1 - \alpha^{-1} z_4^{-1})}{(1 - z_1^{-1} z_2 z_4^{-1}) (1 - z_3 z_4^{-1})} \Bigg|_{z_i =1}.
\end{align*}

From this example, we see an instance of $t_1=\alpha$, being an algebraic number.
\end{enumerate}

Now we turn to the computation of $\CT_s G(s)$.

Firstly, we need to compute the order of $G(s)$, which is given by
$$-d=\theta(G(s))= \theta(L(e^s))+|\bar{B}_0|-|B_0|,$$
as $\theta(1 - e^{b s} t_i) = 0$ for $t_i \neq 1$ and $\theta(1 - e^{b s}) = 1$.

We shall focus on the $d>0$ case, for the $d\leq 0$ case is trivial. Observe that $d_0=\theta(L(e^s))\ge 0$.
It follows that $d_1=|B_0|-|\bar{B}_0|$ has to be positive. Now $d=d_1-d_0> 0$ implies that $d_1>d_0$.

Let $y_i=t_i/(1-t_i)$ and follow notation in Eq.\eqref{e-GTodd}. Then we have
$$\frac{1-t_i}{1-e^s t_i}=\frac{1}{1- \frac{t_i}{1-t_i}(e^s-1)}= \frac{1}{1-y_i (e^s-1)} =f(s,y_i).$$
Now we can rewrite $s^d G(s)$ as a product of power series as follows.
\begin{align*}
 s^d G(s)&=A \cdot E(s) \cdot F(s),
 \end{align*}
 where $A$ (free of $s$), $E(s)$ and $F(s)$ are given by
 \begin{align}
A&= (-1)^{|\bar B_0|+|B_0|}\cdot \frac{\prod_{b \in \bar{B}_0}b}{\prod_{b \in B_0} b} \cdot\prod_{i=1}^r (1-t_i)^{|\bar B_i|-|B_i|} \label{eq-A}
,\\
E(s)&= L(e^s)/s^{d_0} = \sum_{n\geq d_0} \frac{s^{n-d_0}}{n!} \sum_{i=1}^{U} \ell_i a_i^n,   \label{eq-E} \\
F(s)&=\frac{\prod_{b \in {B}_0} f(bs)} {\prod_{b \in \bar{B}_0} f(bs)}
\prod_{i=1}^r \frac{\prod_{b \in {B}_i }f(bs,y_i) } {\prod_{b \in \bar{B}_i} f(bs,y_i) }. \label{eq-F}
\end{align}
Note that the $F(s)$ is exactly the $F(s)$ in \eqref{e-GTodd} when $a=0$.

Let us say that a prime number $p$ is \emph{suitable} for $G(s)$ if $p > d_1=|B_0|-|\bar{B}_0|$  and $B_0$ contains no $0$ in $\Z_p$.
We have the following algorithm.

\textbf{Algorithm CTGTodd}

Input: A suitable prime $p$ and $G(s)= L(e^s) \frac{\prod_{b \in \bar{B}_0} (1-e^{b s})} {\prod_{b \in B_0} (1- e^{b s})}
\prod_{i=1}^r \frac{\prod_{b \in \bar{B}_i } (1- e^{b s} t_i)} {\prod_{b \in B_i} (1- e^{b s}t_i)}$.

Output: $\CT_s G(s)$ modulo $p$.
\begin{enumerate}
\item[S1] Set $d_1=|B_0|-|\bar{B}_0|$ and compute $L(e^s) \xmod{s^{d_1+1}} = \sum_{n=0}^{ d_1} \frac{s^{n}}{n!} \sum_{i=1}^{U} \ell_i a_i^n$. Determine $d_0=\theta(L(e^s))$ and
$d=-\theta(G(s))$. If $d\leq 0$ then return $G(0)$; else  do the following steps.

\item[S2] Set $E(s)\xmod{s^{d+1}}=\sum_{n=0}^{d} E_n s^n= \sum_{n=d_0}^{d_1} \frac{s^{n-d_0}}{n!} \sum_{i=1}^{U} \ell_i a_i^n $.

  \item[S3] Next Compute $F(s) \xmod{s^{d+1}}=\sum_{n=0}^{d} F_n s^n $ by Theorem \ref{theo-GTodd}.

  \item[S4] Return
  \begin{equation}\label{eq-ctGs}
    \CT_s G(s)=[s^d]s^d G(s)= A \sum_{n=0}^d E_n F_{d-n},
  \end{equation}
 where we make the substitutions $y_i=t_i/(1-t_i)$ back and $A$ is given in \eqref{eq-A}.
\end{enumerate}

\begin{prop} \label{prop-CTG}
Under the established notation, if \( p \) is a suitable prime, that is, \( p > d_1 \) and \( B_0 \) contains no multiples of \( p \), then Algorithm CTGTodd correctly computes \( \CT_s G(s) \) over \( \Z_p \) in time \( O((r+1)d^{r+1}\log(d)+\log^2(d) \sum_{i=0}^{r}(|B_i|+|\bar{B}_i|)) \). The output takes the form
$$ \CT_s G(s) = A \sum_{n=0}^{d} E_n \sum_{k_1+k_2+\cdots +k_r \leq d-n} c_{k_1,\dots,k_r}
y_1^{k_1}\cdots y_r^{k_r}, $$
where \( k_i \) are nonnegative integers, and \( c_{k_1,\dots, k_r} \in \Z_p \).
In particular, if \( \bar{B}_i \) are all empty, then \( \CT_s G(s) \) can be computed as a sum of at most \( \binom{d+r}{r} \) simple rational functions, each having a numerator that is a Laurent polynomial with at most \( U \) monomials, and a denominator that is a product of at most \( d+\sum_{i=1}^r ({|B_i|}) \) binomials.
\end{prop}
\begin{proof}
The condition on \( B_0 \) ensures the validity of \( A \) when modulo \( p \). The condition \( p > d_1 \) guarantees the validity of \( E(s) \) and \( F(s) \).

For the first part, Step (S3) dominates the computational time. The time complexity is derived from \eqref{eq-ctGs} and Theorem \ref{theo-RegularExp}, leveraging the fact that \( F(s) \mod{s^{d+1}} \) is regular in \( \mathbf{y} \). The remaining part follows straightforwardly by setting \( y_i = t_i/(1-t_i) \), and we use the fact that the number of nonnegative integer solutions to \( k_1+k_2+\cdots+k_r \leq d \) is equal to \( \binom{d+r}{r} \).
\end{proof}

\begin{rem}\label{rem-caser0}
  The case $r=0$ should be specially treated, due to  its relevance in lattice point counting problems.
   In this scenario, the function $G(s)$ in \eqref{eq_Gtype} reduces to
\begin{align*}
G(s)&=A\cdot \frac{e^{-as}L(e^s)}{s^{d_1}} \cdot F(s;a,B_0,\bar{B}_0),
\end{align*}
where $A=(-1)^{|\bar B_0|+|B_0|}\frac{\prod_{b \in \bar{B}_0}b}{\prod_{b \in B_0} b}$, $d_1=|B_0|-|\bar{B}_0|$
and we select $a=\frac{p_1(\bar B_0)-p_1(B_0)}{2} $ to ensure that
$$ F(s;a,B_0,\bar{B}_0)=e^{as}
 \frac{\prod_{b\in B_0} f(bs)}{\prod_{b\in \bar B_0} f(b s)}$$
is an even function.
We shall employ Theorem \ref{theo-cor-todd} to compute $F(s) \xmod{s^{d+1}}$ instead.
The computational complexity in this case is $O\Big(d\log d+\log^2(d)(|B_0|+|\bar{B}_0|)\Big)$.
\end{rem}

\section{Application to MacMahon's Partition Analysis}\label{s-ApplicationMPA}
In this section, we explain how the constant term of generalized Todd type arise naturally in MPA. Indeed, the basic problem (Problem \ref{p-prob} below) can be transformed into constant term of generalized Todd type.
Thus Algorithm CTGTodd applies to give a significant speed up of existing algorithms like CTEuclid and LattE.

\subsection{The Basic Problem}
The core problem of MPA is to compute the constant term in $\Lambda=\{\lambda_1=x_{m+1},
\dots, \lambda_\ell=x_{m+\ell}\}$ of an Elliott rational function $E=E(x_1,x_2,\dots, x_{m+\ell})$ written in the form:
\begin{equation}\label{f-Elliott}
E = \frac{\text{ a Laurent polynomial}}{
\prod_{j=1}^n (1-\text{Monomial}_j)}=\frac{L}{
\prod_{j=1}^n (1-M_j)}=L\prod_{j=1}^n \Big(\sum_{k\ge 0} M_j^k\Big),
\end{equation}
where we need to assume the convergence of the product, and
 $\CT_{\Lambda} E$ means to  extract all terms free of the $\lambda_i$'s
 in the formal series expansion of $E$.

Many algorithms have been developed. See \cite{de2004effectiveLattE,xin2015euclid,andrews2001macmahonOmegaalgorithm2} and references therein.  We only discuss two of them: one is the Barvinok's algorithm \cite{barvinok1994polynomialalgorithm}
in Computational Geometry, with an implementation by the LattE package in C-language; the other is the
CTEuclid algorithm \cite{xin2015euclid} in Algebraic Combinatorics implemented in Maple. The first step in the CTEuclid algorithm is to add the slack variables
$z_1,z_2,\dots,z_n$ to avoid the multiple roots problem.

\def\d{\bar{d}}

Both algorithms reduces the constant term to the following basic problem.
Let $\d$ be equal to $n-\ell$ and be called the dimension of the problem.

\begin{prob}[Basic Problem]\label{p-prob}
Given  an Elliott-rational function $Q=Q(x_1,\dots x_m; z_1,$ $\dots,z_n)$ written as
a short sum as input:
\begin{align}\label{eq-Bigsum}
 Q= \sum_{i=1}^{N}  \frac{L_i(x_1,\dots x_m;z)}{(1-M_{{i1}}z^{\beta_{i1}})(1-M_{{i2}}z^{\beta_{i2}})\cdots (1-M_{{i\d}}z^{\beta_{i\d}})},
\end{align}
where the $M_{ij}$'s are monomials in $x_1,\dots x_m$, the $L_i$'s are Laurent polynomials with at most $U$ monomials, the $\beta_{ij}$'s are integral vectors.
 Compute its limit at $z_i=1,\ 1\leq i \leq n$, i.e., $Q(x;\mathbf{1})=Q(x;1,\dots,1)$, which is known in advance to be well defined.
\end{prob}

Before go ahead, let us explain this problem by a simple example.
\begin{exam}\label{exam-square}
Let
$P$ be the polytope in the plane with vertices
$(0,0),(1,0),(0,1),(1,1)$. Then $nP$ is the dilated polytope with with vertices
$(0,0),(n,0),(0,n),(n,n)$. Compute the number $i_P(n)$ of lattice points in $nP$ and its generating function $I_P(t)=\sum_{n\ge 0} i_P(n) t^n$.
\end{exam}
\begin{proof}[First Solution]
Geometers like to use Brion's theorem \cite{brion1988points} to write the generating polynomial as a short sum as follows
\begin{align}\label{eq-square}
  Q_n=\frac{1}{(1-x)(1-y)} +\! \frac{x^n}{(1-x^{-1})(1-y)} +\! \frac{y^n}{(1-x)(1-y^{-1})} +\!\frac{x^ny^n}{(1-x^{-1})(1-y^{-1})}.
\end{align}
Taking limit at $x,y\to 1$ gives the desired answer $i_P(n)=(n+1)^2$.

We can also consider the generating function

\begin{align}\label{eq-seriessquare}
  Q(t;x,y)= \sum_{n\ge 0} Q_n& t^n
   = \frac{1}{(1-x)(1-y)(1-t)} + \frac{1}{(1-x^{-1})(1-y)(1-x t)}\\
   &+ \frac{1}{(1-x)(1-y^{-1})(1-yt)} +\frac{1}{(1-x^{-1})(1-y^{-1})(1-x y t)}\notag.
\end{align}
Taking limit at $x,y\to 1$ gives $I_P(t)=\frac{1+t}{(1-t)^3}$.
\end{proof}

However, for complicated problems, the $Q$ might be a sum of millions of simple rational functions, so that
taking limit is not an easy task.

\subsection{Solving the Basic Problem by Algorithm CTGTodd} \label{subsec-GTodd}
Let us see how to compute $Q(t_1,\dots, t_m;\mathbf{1})$ in two major steps over $\Q$:
\begin{enumerate}
  \item[Step 1] Choose a valid vector $\gamma$ and make the substitution
$z_j\rightarrow \kappa^{\gamma_j}$ for all $j$. Then we have only one slack variable $\kappa$.

  \item[Step 2] Make the substitution $\kappa=e^s$ and take constant term in $s$ for each term. 
\end{enumerate}

Step 1 seems unavoidable for two reasons: i) direct substitutions of the $z$'s by $1$ does not work for possible denominator factors like $1-z_1z_2$ in some of the terms;
ii) combining the terms to a single rational function might result in a monster numerator that can not be handled by the computer.

A valid vector $\gamma$ must be picked such that there is no
zero denominator in any term. Let $S=\{\beta_{ij}: M_{ij}=0\}$.
Then we need the inner product $\langle \gamma, \beta_{ij}\rangle \neq 0$ for all $\beta_{ij}\in S$.
Barvinok showed that such $\gamma$ can be picked in polynomial time by choosing points
on the moment curve. De Loera et al. \cite{de2004effectiveLattE} suggested using random vectors to avoid large integer entries.

Working over $\Q$ may meet long digit integer problem for complicated problems. CTEuclid used modulo  $p$ computation by choosing a reasonably large prime $p$.

We also use the above two steps  to solve the basic problem using modulo $p$ computation. In Step 2, we apply Algorithm CTGTodd. The new points are: i) We find a hidden condition on $p$; ii) We give the first reasonable complexity result.

The condition on $p$ hides in Step 2. We use Laurent series expansion to compute $Q(x_1,\dots x_m; \kappa^{\gamma_1},\dots, \kappa^{\gamma_n})|_{\kappa=1}$, where
$$Q(x_1,\dots x_m; \kappa^{\gamma_1},\dots, \kappa^{\gamma_n}) = \sum_{i=1}^N \frac{L_i(x_1,\dots x_m; \kappa^{\gamma_1},\dots, \kappa^{\gamma_n})}{\prod_{j=1}^{\d}(1-\kappa^{\langle \gamma, \beta_{ij}\rangle}M_{ij})}.$$
By making the exponential substitution $\kappa=e^s$, we arrive at
$$Q(x_1,\dots x_m;e^s) = \sum_{i=1}^N  \CT_s \frac{L_i(x_1,\dots x_m; e^s)}{\prod_{j=1}^{\d}(1-e^{s\langle \gamma, \beta_{ij}\rangle}M_{ij})}.$$
This is exactly a sum of constant terms of generalized Todd type in which the case all $\bar{B}_i$ are empty.

In order to apply Algorithm CTGTodd, we need to choose a suitable $p$ so that
a valid $\gamma$ can be found easily. Now the condition
on $\gamma$ becomes $\langle \gamma, \beta_{ij}\rangle \not\equiv 0 \pmod{p}$,
because if otherwise, the corresponding denominator factor will
become $1-\kappa^{kp} \to  1-e^{kps}$, which is $o(s^p)$ when modulo $p$.
See also the condition in Proposition \ref{prop-CTG}. Thus $\gamma$ is a vector in
$\Z_p^n$ that does not belong to $|S|\le N \d$ hyper-planes.
The existence of $\gamma$ is guaranteed when $p$ tends to infinity,
but we need a good bound on $p$.

We claim that if $p\approx |S|/k$, then the probability of a random vector $\gamma$ to be valid is roughly $e^{-k}$. This estimate is based on the fact that for any $\beta_{ij}\in S$, the probability for $\gamma$ to satisfy $\langle \gamma, \beta_{ij}\rangle \not\equiv 0 \pmod{p}$ is clearly $1-1/p$ (except the rare case
when each entry of $\beta_{ij}$ is a multiple of $p$).
Assume the independence of the $|S|$ conditions. Then the probability for a random $\gamma$ to be valid is roughly
$$(1-p^{-1})^{|S|}=\left[(1-p^{-1})^p\right]^{|S|/p} \approx e^{-k}.$$
This means if we choose $p>|S|/3$ then we can expect to obtain a valid $\gamma$
by trying $20\approx e^3$ random vectors.

\begin{prop}
Follow notation above. If $p=O(|S|)$ is a prime, then it is suitable and Problem \ref{p-prob}
for the $m\geq 1$ case can be solved by the above process in time $O(N \d^{\d+1} \log(\d) )$.
\end{prop}
\begin{proof}
The first assertion is obvious. The proof of the second assertion is given below.

Step 1 tries $O(1)$ random vectors $\gamma$ to find a valid one. Each try cost at most $O(|S|n)$ operations in $R$ for testing if $\gamma$ is valid. Thus the cost for this step is $O(|S|n)$.

In step 2, each term is transformed to a constant term of GTodd type. When applying Algorithm CTGTodd, we need to clarify the parameter $d$ and $r$ for each term. These parameters may be different, but they must satisfy $d+r \leq \d$. Assuming the worst case when $|B_0| = d$ and for $i \geq 1$, each $|B_i|$ is equal to one, we have
$r=\d -d$. Thus the complexity is $O((\d-d+1) d^{\d-d+1}\log(d)+ \log^2(d) \d )$
by Proposition \ref{prop-CTG}. From here, it is easy to obtain the complexity in the proposition.
We report that computer experiment suggests that the maximum of $(\d-d+1) d^{\d-d+1}\log(d)$ is attained roughly at $d= \d/3$.
\end{proof}

Let us revisit Example \ref{exam-square}.
\begin{proof}[Second Solution of Example \ref{exam-square}]
For \eqref{eq-square},
instead of taking limit at $x,y \rightarrow 1$, we
set  $x=y=\kappa =e^s$:
$$Q_n(e^s)=\frac{1}{(1-e^s)^2} + \frac{e^{ns}}{(1-e^{-s})(1-e^s)}+\frac{e^{ns}}{(1-e^s)(1-e^{-s})} +\frac{e^{2ns}}{(1-e^{-s})^2}. $$
Taking constant term in $s$ for each term gives
\begin{align*}
   i_P(n)=\CT_s Q_n(e^s) & =\CT_s\frac{1}{(1-e^s)^2} + \CT_s\frac{2e^{ns}}{(1-e^{-s})(1-e^s)} +\CT_s\frac{e^{2ns}}{(1-e^{-s})^2} \\
   & =[s^2]\frac{s^2}{(1-e^s)^2} + [s^2]\frac{-2e^{ns} s^2}{(1-e^{-s})(1-e^s)} +[s^2]\frac{e^{2ns} s^2}{(1-e^{-s})^2}\\
   &=\frac{5}{12}+\Big(\frac{1}{6}-n^2\Big)+ \Big({\frac {5}{12}}+2\,n+2\,{n}^{2}\Big)= 1+2n+n^2.
\end{align*}

We can also compute the generating function in a similar way. For \eqref{eq-seriessquare}  We have
\begin{align*}
 Q(t;e^s)=&\frac{1}{(1-e^s)^2(1-t)} +\frac{1}{(1-e^{-s})(1-e^s)(1-e^{s}t)} \\
 +&\frac{1}{(1-e^{-s})(1-e^s)(1-e^{s}t)} +\frac{1}{(1-e^{-s})^2(1-e^{2s}t)}.
\end{align*}
Taking constant term in $s$ for each term gives
\begin{align*}
 I_P(t)=\CT_s Q_n(t;e^s)&=\CT_s\frac{1}{(1-e^s)^2(1-t)} + \CT_s\frac{1}{(1-e^{-s})(1-e^s)(1-e^{s}t)} \\
   & + \CT_s\frac{1}{(1-e^{-s})(1-e^s)(1-e^{s}t)}+\CT_s\frac{1}{(1-e^{-s})^2(1-e^{2s}t)}\\
   &={\frac {5}{12\left(1-t\right)}}+ {\frac{5{t}^{2}+8t-1}{12\left(-1+t\right)^{3}}}+ {\frac{5{t}^{2}+8t-1}{12\left(-1+t\right)^{3}}}-{\frac{5{t}^{2}+38t+5}{12\left( -1+t \right)^{3}}}\\ &=\frac{1+t}{(1-t)^3}.
\end{align*}
\end{proof}

\subsection{Lattice Point Counting}
The special case of \( m = 0 \) in Problem \ref{p-prob} is intimately connected to the problem of counting lattice points within a polytope.
We present the following complexity result.
\begin{prop}
Suppose \( f(P;\mathbf{y}) \) is defined as in \eqref{Barvinok-short-sum}, with \( N = |I| \). For a prime \( p = O(N \d) \), the computation of \( f(P;\mathbf{1}) \mod{p} \) can be achieved in time \( O(N \log^2(\d)\d) \).
\end{prop}
\begin{proof}
Given \( p = O(N \d) \), we can select a valid \( \gamma \) in time \( O(N \d) \). This allows us to perform the substitution and arrive at a sum of \( N \) constant terms of the Todd type. By Remark \ref{rem-caser0}, the complexity of each constant term is \( O(\d \log(\d) + \log^2(\d)\d) \) over \(\mathbb{Z}_p\).
In summary, the total complexity is \( O(N \log^2(\d)\d) \).
\end{proof}
When computing \( f(P;\mathbf{1}) \), LattE's default method involves identifying a valid \( \gamma \) and then utilizing the substitution \( \kappa \to (1 + s) \). This traditional approach results in a complexity of \( O(\d^3) \) for each term, leading to a total complexity of \( O(N \d^3) \) over \(\mathbb{Q}\). However, this method may encounter issues with large integer numbers. A complexity result, approximately as high as \( O(N \d^6 \log^3(\d)) \), was presented in \cite{de2009ehrhartTodd} by considering the binary encoding length of the rational numbers involved.

If the precise value of \( f(P;\mathbf{1}) \) is of interest, a sound strategy would be as follows: first, provide a reasonable estimate for \( f(P;\mathbf{1}) \), and then utilize our method to compute \( f(P;\mathbf{1}) \mod{p_i} \) for some large primes \( p_i \). Finally, reconstruct \( f(P;\mathbf{1}) \) using the Chinese remainder theorem.

\subsection{Ehrhart Series for Magic Squares}\label{App-MagicSquares}
Consider a bounded rational polytope \( P = \{\alpha : A\alpha = b, \alpha \geq 0\} \subset \mathbb{R}^n \), the function
$$ i_P(k) := \# (kP \cap \mathbb{Z}^n) = \# \{\alpha \in \mathbb{Z}^n : A\alpha = kb, \alpha \geq 0 \}$$
for any positive integer \( k \) was first analyzed by E. Ehrhart \cite{ehrhart1962polyedres}. It is referred to as the Ehrhart polynomial when the vertices of \( P \) are integral and is known as the Ehrhart quasi-polynomial for arbitrary rational convex polytopes \cite[Ch. 4]{stanley2011EC1}.
The Ehrhart series of \( P \), denoted by
$$I_P(t) := \sum_{k\ge 0} i_P(k) t^k,$$
is recognized as an Elliott-rational function in the single variable \( t \).

A representation of the constant term of \( I_P(t) \) is well-established. Consequently, we can employ the constant term method to generate a big sum as in \eqref{eq-Bigsum} with \( m = 1 \) and compute \( I_P(t) \) directly. This approach appears superior to the conventional method that relies on Lagrange's interpolation.

Here, we confine our attention to the magic square polytope \( MS_n \) of order \( n \), which is defined by the following linear constraints:

\begin{align*}
  a_{i,1}+a_{i,2}+\cdots+a_{i,n}=1,   \text{ for }1\le i\le n \\
  a_{1,j}+a_{2,j}+\cdots+a_{n,j}=1,    \text{ for }1\le j\le n \\
  a_{1,1}+a_{2,2}+\cdots +a_{n,n}=1, \ a_{n,1}+a_{n-1,2}+\cdots +a_{1,n}=1.
\end{align*}

In other words, the lattice points in \( s MS_n \) are referred to as order \( n \) magic squares with magic sum \( s \). These are \( n \times n \) nonnegative integer matrices \( M = (a_{i,j})_{n \times n} \) whose row sums, column sums, and both diagonal sums are all equal to \( s \).

The Ehrhart series for \( MS_n \) is given by
$$ I_{MS_n}(t) = \CT_{\lambda, \mu, \nu} \prod_{1 \leq i,j \leq n} \frac{1}{1 - \lambda_i \mu_j \nu_1^{\chi(i = j)} \nu_2^{\chi(i + j = n + 1)}} \times \frac{1}{1 - t(\lambda_1 \cdots \lambda_n \mu_1 \cdots \mu_n \nu_1 \nu_2)^{-1}}. $$
Up to this point, the Ehrhart series \( I_{MS_n}(t) \) has only been computed for \( n \leq 6 \). The computation of \( I_{MS_6}(t) \) was initially performed by the CTEuclid package in \cite[Section 5.2]{xin2015euclid}.

With the Todd computation developed in this paper, we can recompute \( I_{MS_6}(t) \) with a substantial speedup.
The computation by CTEuclid was divided into two steps: i) by taking the constant term in \( \lambda, \mu, \nu \), we obtain
$$ I_{MS_6}(t) = \sum_i Q_i(t, e^s) \Big|_{s=0} = \sum_i \CT_s Q_i(t; e^s); $$
ii) compute these constant terms separately to obtain \( I_{MS_6}(t) \) for each of the three large prime numbers \( p_1, p_2, p_3 \), and reconstruct \( I_{MS_6}(t) \) using the Chinese remainder theorem.
In the second step, the running time was approximately \( 70 \) days for each \( p_j \). We utilize our log-exponential technique to expedite the second step, employing the same \( Q_i \)'s and \( p_j \)'s in Maple. The running time is reduced from \( 70 \times 3 \) days to approximately \( 3 \times 1 \) days. The running time would be further reduced if we were able to implement the ideas outlined in Subsection \ref{ssec-complexity}.

\section{Application to ILP Utilizing Barvinok's Counting Algorithm}\label{s-ApplicationILP}
In this section, we delve into the realm of integer linear programming (ILP), denoted as \( \max(c,P) \), which is defined as follows:
\[
\max(c,P) := \max \left\{ c \cdot x : x \in P \cap \mathbb{Z}^{\d} \right\}, \quad P = \left\{ x \in \mathbb{R}^{\d} : Ax \leq b, x \geq 0 \right\},
\]
where \( A \in \mathbb{Z}^{m \times \d} \), \( b \in \mathbb{Z}^m \), and \( c \in \mathbb{Z}^{\d} \) constitutes the cost vector. We presume that the polytope \( P \) is both bounded and non-empty, rendering it a discrete counterpart to the well-established linear programming (LP) paradigm. The study of ILP is foundational and underpins numerous crucial applications. For a comprehensive overview of ILP, readers are directed to the seminal works \cite{wolsey1999integer,wolsey2020integer,schrijver1998theory}.

\subsection{The BSCT Algorithm: A Binary Search Technique Leveraging Constant Term Manipulations}\label{s-sub-BSCT}
We introduce a polynomial-time algorithm, BSCT, for addressing ILP when the dimension \( \d \) is fixed. This algorithm employs Barvinok's short rational function \( f(P;\mathbf{y}) \), as presented in \eqref{Barvinok-short-sum}. Throughout this discussion, when we refer to efficient computation, we imply that the running time and the number of terms \( |I| \) are both polynomially bounded in terms of the input size.

Our algorithm is different from the  two existing algorithms developed for this purpose in the literature: Lasserre's heuristic in \cite{lasserre2004integer} and its enhanced variant, called the digging algorithm in \cite{de2004three}. Additionally, it differs from the binary search method proposed by Barvinok in \cite{barvinokPommersheim1999summary}.

Our algorithm's fundamental approach involves two distinct steps. Firstly, it seeks to identify \( \max(c,P) \), the maximum value attainable by the cost vector \( c \) within the polytope \( P \). Subsequently, it aims to determine the vector \( \alpha \) that achieves this maximum. This method contrasts with the digging algorithm, which concurrently searches for both \( \alpha \) and \( \max(c,P) \). However, in our current implementation, we have limited our focus to the initial step.

In the course of our development, it is advantageous to examine the following equivalent problem:
\[
\min(c,P) := \min \left\{ c \cdot x : x \in P \cap \mathbb{Z}^{\d} \right\}, \quad P = \left\{ x \in \mathbb{R}^{\d} : Ax \leq b, x \geq 0 \right\}.
\]
It is evident that \( \min(c,P) = -\max(-c,P) \), thus establishing a direct relationship between the minimum and maximum objectives.

Let us delineate the operational framework of the digging algorithm. For a given cost vector \( c \), we substitute \( y_k = y_k t^{c_k} \) into \( f(P;y) \) to yield
$$g(P;\mathbf{y},t) = \sum_{i\in I} g_i(\mathbf{y},t) = \sum_{i\in I} \epsilon_i \frac{y^{u_i}t^{cu_i} }{\prod_{j=1}^{\d} (1-y^{v_{ij}} t^{cv_{ij}} )},$$
which, though a Laurent polynomial, can be treated as a Laurent series in \( t \). Consequently,
$$\min(c,P) = \theta_t(g(P;\mathbf{y},t)) \ge \m = \min \left\{ \theta_t g_i(t) \right\},$$
where \( \theta_t g(t) \) represents the order of the Laurent series \( g(t) \) in \( t \), as defined in Section \ref{s-ConstantTermsGTodd}.

If the easily computed coefficient \([t^m] g(P;y,t) \) is nonzero, then \( \m = \min(c,P) \), satisfying Lasserre's criterion;
Otherwise if \([t^m] g(P;y,t) \) is zero, one must ascertain the least degree term, which encapsulates the essence of the digging algorithm as described.

Our contention is encapsulated in the following theorem.
\begin{theo}\label{theo-opt}
For any given polytope \( P \) and cost vector \( c \), there exists a polynomial-time algorithm that can express
\begin{align}\label{eq-opt}
g(P;\mathbf{1},t) = \sum_{i \in I'} h_i(t) = \sum_{i \in I'} \varepsilon'_i \frac{t^{a_i}}{\prod_{1 \le j\le d_i}(1-t^{b_{i,j}}) },
\end{align}
where \( \varepsilon'_i \in \mathbb{Q} \), \(b_{i,j}>0,\; d_i \leq \d \) for all \( i \), with the convention \(\prod_{1 \le j\le 0}(1-t^{b_{i,j}}) := 1\). Moreover, \( \min(c,P) \) can be determined in polynomial time through a binary search algorithm.
\end{theo}
\begin{proof}
We commence by computing \( f(P;\mathbf{y}) \) using Barvinok's algorithm, thereby obtaining \( g(P;\mathbf{y},t) \). Next, we identify a vector \( \lambda \) as per the methodology outlined in Subsection \ref{subsec-GTodd} to yield
$$g(P;\mathbf{1},t) = \CT_s \sum_{i\in I} \epsilon_i \frac{e^{\lambda u_i s}t^{cu_i} }{\prod_{j=1}^{\d} (1-e^{\lambda v_{ij} s} t^{cv_{ij}} )} = \sum_{i\in I} \epsilon_i \CT_s \frac{e^{\lambda u_i s}t^{cu_i} }{\prod_{j=1}^{\d} (1-e^{\lambda v_{ij} s} t^{cv_{ij}} )}.$$
The first portion of the theorem is then established by Proposition \ref{prop-CTG}.

Regarding the second part, we commence by utilizing LP to resolve \(\max \{c\cdot \alpha : \alpha \in P\} - \m\). This yields an upper bound \( UB \) for the degree in \( t \) of \( G(t) = t^{-\m} g(P;\mathbf{1},t) \), which exclusively possesses nonnegative powers of \( t \). The lower bound \( LB \) is assigned the value of \( 0 \). We then employ a standard dichotomy procedure to either reduce \( UB \) or augment \( LB \) until \( LB = UB \).
To execute the dichotomy procedure, we note that \( G^k = \sum_{i=0}^k G_i \) is efficiently computable. Specifically, we have \( G^k = \CT_t t^{-k} G(t)/(1-t) \), thus
$$G^k = \CT_t t^{-k} G(t)/(1-t) = \sum_{i\in I'} \CT_t t^{-k} \cdot t^{-\m} h_i(t)/(1-t) $$
is efficiently computable (utilizing Barvinok's counting algorithm) for each term. Consequently, for any \( k' > k \), \( G^{k'} - G^k = \sum_{i=k}^{k'-1} G_i \) is efficiently computable and is nonzero if and only if \( G_i > 0 \) for some \( k \leq i < k' \).
\end{proof}
\begin{rem}
The dichotomy procedure proposed by Barvinok necessitates the enumeration of lattice points within \( P(k) = \{\alpha \in P : c\cdot \alpha \geq k\} \) to refine the range for the maximum value of \( c\cdot \alpha \). However, the performance of this approach does not appear to be fast (as evidenced by the BBS column in Table \ref{table-compute time}). A plausible explanation for the slow performance is that the newly defined polytope \( P(k) \) may introduce new vertices, thereby complicating \( P(k) \) relative to \( P \). For instance, in Figure \ref{BBSpic}, the coordinates of \( v_1, v_2 \) might be rational numbers with extended decimal representations, even if \( P \) is integral.
\end{rem}

\begin{figure}[!ht]
$$
  \hskip .1 in \vcenter{ \includegraphics[height=1.4 in]{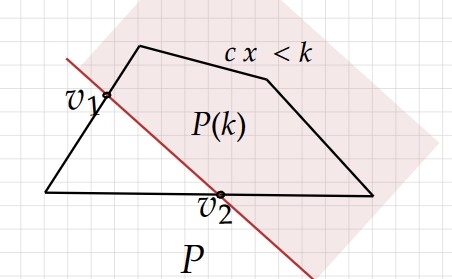}}
$$
\caption{The polytope $P(k)$ and its new vertices $v_1$ and $v_2$.
\label{BBSpic}}
\end{figure}

Our method necessitates the resolution of polynomially many knapsack problems. In the worst-case scenario, these can be addressed in polynomial time by leveraging Barvinok's counting algorithm. In practice, we address scenarios where \( \min(c,P) - \m \) is small, in which case a direct search will swiftly yield a solution. This insight motivates the following variant of the binary search algorithm.

\textbf{Algorithm BSCT}

Input: A bounded polytope \( P \) specified by \( A \) and \( b \), and a cost vector \( c \) as described.

Output: \(\min(c,P)\).

\begin{enumerate}
 \item[S1] Compute \( f(P;\mathbf{y}) \) as a short sum in \eqref{Barvinok-short-sum} using LattE.

 \item[S2] Compute \( g(P;\mathbf{1},t) = \sum_{i \in I'} h_i(t) \) as in \eqref{eq-opt} employing Theorem \ref{theo-opt}.

 \item[S3] Set \( \m = \min \{ \theta_t h_i(t) \} \) and define \( G(t) = t^{-\m} g(P;\mathbf{1},t) = \sum_{i \in I'} t^{-\m} h_i(t) \).

 \item[S4] Compute \( G(t) \xmod{ t^{k_0}}\) with the default value \( k_0 = 2^{16} \) if \( \bar{d} \leq 16 \). If the result is nonzero, then \(\min(c,P) = \theta_t(G(t) \xmod{ t^{k_0}}) \) has been determined.

 \item[S5] If \( G(t) \xmod{ t^{k_0}} \) is zero, iteratively compute \( G^{k_i} \) for \( i = 1, 2, \dots \) with \( k_i = 2^i k_0 \) until \( G^{k_{i'}} > 0 \). Then \( 2^{i'-1}k_0 \leq \min(c,P) - \m \leq 2^{i'}k_0-1 \). Proceed with a standard binary search to ascertain \(\min(c,P)\).
\end{enumerate}
The default value \(k_0 = 2^{16}\) in Step 4 is chosen for its practical efficiency in computing a typical term in \( G(t) \xmod{ t^{k_0}} \):
 \begin{align*}
   h(t) \xmod{ t^{k_0}} &= \frac{t^{a}}{(1-t^{b_1})(1-t^{b_2})\cdots (1-t^{b_{{d}}})} \xmod{ t^{k_0}}\\
    &= \frac{1}{(1-t^{b_1})(1-t^{b_2})\cdots (1-t^{b_{{d}}})} \xmod{ t^{\ell}} \cdot t^a,
 \end{align*}
where $d\leq \d$,  \( b_i \) are positive integers, \( 0 \leq a < k_0 \) and \( \ell = k_0 - a \leq k_0 \). From Example \ref{}, we observe that \( t^{-a}h(t) \xmod{ t^{\ell}} \) can be computed in time \( O(\ell \log(\ell)) \). In practice, \( G(t) \xmod{ t^{k_0}} \) is computed swiftly (utilizing the series command in Maple) when \( k_0 = 2^{16} \) in our computer experiments.

Algorithm BSCT is polynomial in nature when the dimension \( \d \) is fixed. The most computationally intensive step is the iteration process. The number of iterations is bounded by \( 2\log(UB-LB) \). Each iteration requires solving \( |I'| \) (which is polynomial in size) knapsack problems, and each knapsack problem can be solved efficiently.

\subsection{Computational Experiments}

In this section, we present the outcomes of our computational experiments aimed at addressing challenging knapsack problems, as presented in \cite{aardal2002hard}. The data utilized in our experiments is detailed in Table \ref{table-knapsack problems}.
For each row in the table, the objective is to maximize \( c \cdot \mathbf{x} \) subject to the constraints \( \mathbf{a} \mathbf{x} = b \) and \( \mathbf{x} \in \mathbb{N}^{\d} \). For the sake of comparison, the cost vector \( c \) is identical to that used in \cite{de2004three}. Specifically, \( c \) is chosen to be the first \( \d \) components of the vector
$$(213,-1928,-11111,-2345,9123,-12834,-123,122331,0,0).$$

It is crucial to recognize that although maximization is equivalent to minimization, the computational time required for \( \max(c,P) \) and \( \min(c,P) \) can vary significantly for a specific pair \( (c,P) \). We employ the BSCT algorithm to resolve both problems; however, we do not provide the corresponding optimal solutions.

In our implementation of BSCT, the computation of \( G^{k_i} \) in the final step is conducted using the CTDenumerant package developed in \cite{LLLCTEuclid}. This substitution was necessitated by the unfortunate discontinuation of LattE's development. The computational times reported in Table \ref{table-compute time} are generated using Maple, with the exception of the initial computation of \( f(P;\mathbf{y}) \), which is produced by LattE. The timings for the Digging and BBS algorithms are sourced from \cite{de2004three}.

\begin{table}[ht]
\renewcommand{\arraystretch}{1.2}
\resizebox{.8\columnwidth}{!}{
\begin{tabular}{|l|l|l|}
\hline
Problem & \multicolumn{1}{|c}{a}                                           & \multicolumn{1}{c|}{b} \\ \hline
cuww1   & 12223 12224 36674 61119 85569                                    & 89643482               \\ \hline
cuww2   & 12228 36679 36682 48908 61139 73365                              & 89716839               \\ \hline
cuww3   & 12137 24269 36405 36407 48545 60683                              & 58925135               \\ \hline
cuww4   & 13211 13212 39638 52844 66060 79268 \ 92482                        & 104723596              \\ \hline
cuww5   & 13429 26850 26855 40280 40281 53711 \ 53714 \ 67141                  & 45094584               \\ \hline
prob1   & 25067 49300 49717 62124 87608 88025 \ 113673 119169                & 33367336               \\ \hline
prob2   & 11948 23330 30635 44197 92754 123389 136951 140745               & 14215207               \\ \hline
prob3   & 39559 61679 79625 99658 133404 137071 159757 173977              & 58424800               \\ \hline
prob4   & 48709 55893 62177 65919 86271 87692 102881 109765                & 60575666               \\ \hline
prob5   & 28637 48198 80330 91980 102221 135518 165564 176049              & 62442885               \\ \hline
prob6   & 20601 40429 40429 45415 53725 61919 64470 69340 78539     95043  & 22382775               \\ \hline
prob7   & 18902 26720 34538 34868 49201  49531   65167  66800 84069  137179 & 27267752               \\ \hline
prob8   & 17035 45529 48317 48506 86120 100178 112464 115819 125128 129688 & 21733991               \\ \hline
prob9   & 3719  20289  29067 60517 64354  65633  76969 102024 106036 119930 & 13385100               \\ \hline
prob10  & 45276 70778 86911 92634 97839 125941 134269 141033 147279 153525 & 106925262              \\ \hline
\end{tabular}
}

\caption{Knapsack problems.} \label{table-knapsack problems}
\end{table}

\begin{table}[ht]
  \centering
  \resizebox{.9\columnwidth}{!}{
  \begin{tabu}{|c|c|l|l|l|l|[2pt]c|c|}
\hline
 Problem & \begin{tabular}[c]{@{}l@{}}Maximum\\Value\end{tabular}  & \begin{tabular}[c]{@{}l@{}}Runtime for \\BSCT \end{tabular}& \begin{tabular}[c]{@{}l@{}} Runtime for \\Digging(Original)\end{tabular} & \begin{tabular}[c]{@{}l@{}}Runtime for \\Digging(S.Cone)\end{tabular} 			& \begin{tabular}[c]{@{}l@{}} Runtime\\for      	BBS\end{tabular} & \begin{tabular}[c]{@{}l@{}}Minimum\\  Value\end{tabular}  & \begin{tabular}[c]{@{}l@{}}Runtime\\ for BSCT \end{tabular} 	 \\ \hline
    cuww1 & 1562142 & $<$0.01 sec.      		& 0.4 sec. & 0.17 sec.     					& 414 sec.      				& 1562142      		 & 0.18 sec. \\\hline

    cuww2   & -4713321   	 	& 35.78 sec.					     	& \textgreater{}3.5h     & \textgreater{}3.5h         & 6600 sec.        			& -4713321  	 	& $<$0.01 sec.          \\\hline
    cuww3   & 1034115     	& \textless{}0.01 sec. 		        & 1.4 sec.      				& 0.24 sec.                 		 & 6126 sec.      		& -4697622      	          & 1357.38 sec.     				\\ \hline
cuww4   & -29355262       		& \textless{}0.01 sec.      & \textgreater{}1.5h     & \textgreater{}1.5h          & 38511 sec.            	& -29355262      & \textless{}0.01 sec.  	\\ \hline
cuww5   & -3246082       	& \ \ --                          	  & \textgreater{}1.5h   & 147.63 sec.                  & \textgreater{}80h      	& -3246082      	   & 0.124 sec.                		\\ \hline
prob1   & 9257735          	 & 3832 sec.        		  & 51.4 sec.                & 18.55 sec.                   & \textgreater{}3h     		& -4568716      		     & 0.5 sec.                      		 \\ \hline
prob2   & 3471390          	 & 0.48 sec.                  	    & 24.8 sec.                & 6.07 sec.                     & \textgreater{}10h   	& -968323        				   & 0.5 sec.                     				 \\ \hline
prob3   & 21291722        	& 0.54 sec.                 		    & 48.2 sec.                & 9.03 sec.                     & \textgreater{}12h   	  & -7757086    		        & 0.6 sec.                	 			 	 \\ \hline
prob4   & 6765166          	& 0.57 sec.           				    & 34.2 sec.               & 9.61 sec.                     & \textgreater{}5h     	  & -10798906    			       & 0.6 sec.                 		     	\\ \hline
prob5   & 12903963        	  & 0.5  sec.          				     & 34.5 sec.               & 9.94 sec.                     & \textgreater{}5h     	 & -6064042      		         & 0.9 sec.               			      	\\ \hline
prob6   & 2645069          	 & 2.6 sec.            				    & 143.2 sec.              & 19.21 sec.                   & \textgreater{}4h    	 & -328675       		        & unavailable                   		   	  \\ \hline
prob7   & 22915859        	 & 0.96 sec.              			     & 142.3 sec.              & 12.84 sec.                   & \textgreater{}4h    	 & unkonwing      		       & --                    					  \\ \hline
prob8   & 3546296            & 5.8 sec.              				     & 469.9 sec.              & 49.21 sec.                   & \textgreater{}3.5h 	& -1285293      		        & 10.2 sec.              			   	  \\ \hline
prob9   & 15507976        	 & 7887.645  sec. 			     & 1408.2 sec.             & 283.34 sec.                  & \textgreater{}11h 	 & -4896754         	         & 115.222 sec.      			   	   \\ \hline
prob10  & 47946931       	& 43.92 sec.            			       & 250.6 sec.              & 29.28 sec.                   & \textgreater{}11h  	   & -12314184      	          & 18.00 sec.         				  \\ \hline
\end{tabu}
}

  \caption{Optimal values and running times for the knapsack problems.}  \label{table-compute time}
\end{table}

We conclude this section by outlining potential enhancements to our BSCT algorithm.
\begin{enumerate}
\item Brion's theorem provides vertex cones whose generators are independent of \( k \) for each expression of the form \( \CT_t t^{-k} \cdot t^{-\m} h_i(t)/(1-t) \). Consequently, the data for Barvinok's simplicial cone decomposition can be reused for various values of \( k \).

\item The expectation (or variance) of the objective function can be computed using the formula \( \frac{d}{dt} g(P;\mathbf{1},t)\big|_{t=1} \). These computations may offer insights into estimating a suitable lower bound for \( \min(c,P) \).

\item A single cone version of the Digging algorithm is proposed in \cite{de2004three}. This concept could be integrated into our method, provided we have a reliable bound for \( \min(c,P) \), allowing us to disregard vertex cones that do not contribute to the optimal solution.
\end{enumerate}

\noindent
{\small \textbf{Acknowledgements:}
The authors wish to extend their gratitude to Shaoshi Chen for valuable suggestions, to Feihu Liu for careful reading of the draft and identifying several errors. The authors are also appreciative of the insightful comments and suggestions offered by the anonymous referees. This work was partially supported by the National Natural Science Foundation of China [12071311].

\end{document}